\documentclass[leqno,12pt]{article} 
\usepackage{amsmath, amssymb, bm}
\usepackage{amsthm} 
\usepackage{graphicx,color}

\theoremstyle{plain} 
\newtheorem{theorem}{\indent\bf Theorem}[section] %
\newtheorem{lemma}[theorem]{\indent\bf Lemma}

\newtheorem{proposition}[theorem]{\indent\bf Proposition}

\theoremstyle{definition} %
\newtheorem{definition}[theorem]{\indent\bf Definition}
\newtheorem{remark}[theorem]{\indent\bf Remark}

\newcommand{\N}{\mathbb{N} }

\newcommand{\R}{\mathbb{R} }

\newcommand{\LA}{\left \langle }
\newcommand{\RA}{\right  \rangle }
\newcommand{\LB}{\left \lbrack }
\newcommand{\RB}{\right  \rbrack }
\newcommand{\LC}{\left ( }
\newcommand{\RC}{\right ) }
\newcommand{\LD}{\left \{ }
\newcommand{\RD}{\right \} }

\newcommand{\DS}{\displaystyle }

\makeatletter
\def\address#1#2{\begingroup
\noindent\parbox[t]{7.8cm}{%
\small{\scshape\ignorespaces#1}\par\vskip1ex
\noindent\small{\itshape E-mail address}%
\/: #2\par\vskip4ex}\hfill%
\endgroup}%
\makeatother
\title{\uppercase{Periodic Problem for Doubly Nonlinear Evolution Equation}} %
\author{
\textsc{Masahiro Koike, Mitsuharu \^{O}tani, and Shun Uchida} %
\date{} %
}

\begin{document}

\maketitle

\footnote{ 
2010 \textit{Mathematics Subject Classification}.
Primary 47J35; Secondary 35B10, 35K55, 35K92.
}
\footnote{ 
\textit{Key words and phrases}. Doubly nonlinear evolution equation,
time-periodic problem, subdifferential, $p$-Laplacian, doubly nonlinear parabolic equation.
}
\footnote{ 
$^{*}$ M. \^{O}tani was supported by the Grant-in-Aid for Scientific Research (C) {\#}18K03382, JSPS Japan.

$^{**}$ S. Uchida was supported by the Grant-in-Aid for Scientific Research (C) {\#}18K03382, JSPS Japan and 
Fund for the Promotion of Joint International Research (Fostering Joint International Research (B))
{\#}18KK0073, JSPS Japan.
}

\begin{abstract}
We are concerned with  the time-periodic problem of some doubly nonlinear equations
governed by differentials of two convex functionals over 
uniformly convex Banach spaces.
Akagi--Stefanelli (2011) \cite{AS_Cau} considered Cauchy problem of the same equation
via the so-called WED functional approach.
Main purpose of this paper is to show the existence of the time-periodic solution 
under the same growth conditions on functionals and differentials as those imposed in \cite{AS_Cau}.
 Because of the difference of nature between Cauchy problem and the periodic problem, 
we can not apply the WED functional approach directly, so we here 
adopt standard compactness methods with suitable approximation procedures.
\end{abstract}

\section{Introduction} 
Let $V$ be a uniformly convex real  Banach space
and $V ^ *$ be its uniformly convex dual.
In this paper,
we are concerned with the following
time-periodic problem of doubly nonlinear  evolution equation: 
\begin{equation*}
  \text{\rm (AP)} 
    \begin{cases}
                ~ d \psi ( u '  (t) ) + \partial \phi (u (t) ) \ni f (t),
			       ~~ & t \in (0,T ) ~~\text{ in } V ^{*} , 
\\
                ~ u(0) = u(T), 
                   ~~ & 
    \end{cases}
\end{equation*}
where $ d \psi $ and $\partial \phi $ are G\^{a}teaux differential and subdifferential 
  of convex functionals $\psi$ and $ \phi$ which are mapping from $V$  into
     $ ( - \infty , + \infty ] $.
       Here and henceforth, $u'$ denotes the time derivative of $u$ and 
         $f$ is a given external force belonging to $ L^{p'}(0,T;V^{*})$,
           where $p ' := p / ( p-1 )$  and $p \in (1, \infty ) $ 
             (precise definitions and assumptions will be given in the next section).
    A typical example which can be reduced to (AP) is given by the following 
       doubly nonlinear parabolic equation:
\begin{equation}
    \alpha (u'(x,t)) -\Delta_m u(x,t )  =  f(x ,t ) ,
\label{Int01} 
\end{equation}
where $\alpha  : \R \to \R $ is a non-decreasing function
and $\DS \Delta_m u :=  \nabla \cdot \LC |\nabla u |^{m-2} \nabla u  \RC$ (so-called $m$-Laplacian).

 Cauchy problem for (AP) in Hilbert space   
  has been studied by  Barbu \cite{Bar0}, where the elliptic regularization technique by adding
   the term $- \varepsilon ( d_V \psi (u' )) '$ is employed 
     and by Arai \cite{Arai}, Senba \cite{Senba}, Colli--Visintin \cite{CV} via 
       the standard relaxation procedure  
         with the additional term $ \varepsilon u' $. 
 Stefanelli \cite{Stef} discussed
  a variational characterization of solution to gradient system 
    relying on the  Br\'ezis--Ekland principle.
 Investigation of global solvability in Banach spaces began with Colli \cite{C},
   where the time discretization and the polygonal chain approximation are exploited.
 In Akagi--Stefanelli \cite{AS_Cau},
  they adopted another approach based on the fact that
   the solution to Cauchy problem for  (AP) with elliptic regularization
    can be characterized by the minimizer of some functional 
     with the weight $\exp (- t / \varepsilon )$,
      the so-called Weighted Energy-Dissipation (WED, for short) functional
 (see also 
    \cite{A-S0} \cite{A-S2} \cite{A-S3} \cite{CO} \cite{MO} \cite{MS1} \cite{MS2}
      and references therein).
 By this procedure, 
   they proved the existence of global solution to Cauchy problem
    assuming some growth  conditions on $\psi $ and $\phi$ (see Remark \ref{Compare} below).

 We here comment on the following variant type of doubly nonlinear equation,
   studied more vigorously than (AP):
\begin{equation}
 (A u (t)) ' + B u(t) \ni  f( t ) ,
\label{Int02} 
\end{equation}
     where $A$ and $B $ are maximal monotone operators.
 Cauchy problem of \eqref{Int02} is investigated in, e.g.,
  \cite{DS} \cite{G} \cite{Hokkanen1} \cite{KP}.
 Akagi--Stefanelli attempt  the WED functional approach to \eqref{Int02} in \cite{A-S1}.
  Moreover, papers
   \cite{AL} \cite{Bernis} \cite{FT} \cite{Ishige} \cite{Ivanov} \cite{Raviart} 
     \cite{Saa} \cite{Tsutsumi} \cite{Ves} are devoted to 
        the initial boundary value problem 
of specific parabolic PDEs  such as
$\partial _t u  - \Delta _m u ^p = f $  and  $\partial _t u ^ p - \Delta _m u  = f $.

Compared with Cauchy problem, there are a few results 
for the existence of time-periodic solution to doubly nonlinear equation.
Periodic problem for \eqref{Int02} is considered in \cite{AH2} \cite{Hokkanen3} \cite{KK1} \cite{KK2}
and concrete PDEs of the same type as \eqref{Int02}
in \cite{FMNP} \cite{FNP} \cite{SYW} \cite{WG} \cite{WY} \cite{WYK}. 
 However, 
to the best of our knowledge, the investigation into  other types of equations
 different from \eqref{Int02}  
can be found  only in Akagi--Stefanelli \cite{AS_Per}.
They consider the solvabilty and structural stability of periodic problem
of the following abstract equation:
\begin{equation}
 A ( u ' (t) )  + \partial \phi ( u(t) )  \ni  f( t ) ,
\label{Int03} 
\end{equation}
where $A$ is a  possibly  multi-valued maximal monotone operator.

 In \cite{AS_Per},
  they restrict their discussion to the Hilbert space framework 
   and impose the linear growth  condition on $A$.  
     These conditions seem to be more restrictive 
       than those imposed for Cauchy problem  in  \cite{AS_Cau} when $A = d \psi $.
 Main purpose of this paper is to show that we can discuss the periodic problem for (AP) 
  under almost the same growth conditions as those in \cite{AS_Cau}.  
 In the next section,
   we fix  several notations, present a few auxiliary tools, and 
     state  our main results more  precisely.
 Section 3 is devoted to our proofs. 
  In Section 4,
   we show that estimates established in the previous section
    are immediately applicable to the study for some structural stability of (AP).
 Finally,  we exemplify the applicability of our setting 
 	by dealing with doubly nonlinear parabolic equation \eqref{Int01}.

 Some parts of our arguments for estimates and convergence rely on \cite{AS_Cau}.
  However, our main strategy seems to be indispensable in order to cope with difficulties
   arising from the difference of nature between Cauchy problem and the time-periodic   
    problem,  for which one would perceive that 
     it is hard to find a suitable variational structure 
      to apply the WED functional approach by \cite{AS_Cau}.
  So here we depart from  the WED functional setting 
   and adopt the standard compactness method
     with a suitable approximation procedure 
      (see (AP)$_{\varepsilon }$ in Theorem \ref{Th1} given later).
 In establishing a priori estimates for solutions of approximate equations, the coercivity of 
   $\partial \phi $  to be assumed plays an essential role,  
     since $u (0) $ is an unknown value which may possibly depend 
       on approximation parameters  for the periodic problem.

  In proving our main results, we first deal with the easier case 
   where $\phi$ dominates $\psi$ in 
    Section 3.2--3.4 and the other case will be treated in Section 3.5.


\section{Main Results}

\subsection{Preliminary}

We first fix some notations and recall basic properties (see  \cite{Bar1}\cite{Bar2}\cite{BRE}), 
which will be used later.

Let $E$ be a real Banach space and $E^{\ast }$ be its dual.
The norms of $E$ and $E ^{\ast }$ are denoted by $|\cdot | _E  $ and  $|\cdot | _{E^{\ast} }  $,
respectively, and the duality pairing by $\LA  u ^{\ast } , u \RA _E $,
where $\LB u , u^{\ast } \RB \in E \times   E^{\ast } $.
We define the duality mapping $F_E $ of $E$ by
\begin{equation*}
 F_E (u) := \LD u^{\ast } \in  E^{\ast };~ \LA  u ^{\ast } , u \RA _E  = |u| ^2 _E = |u^{\ast}| ^2 _{ E ^{\ast}} \RD .
\end{equation*}
By virtue of the Hahn-Banach theorem,
$F_E (u) $ is non-empty for every $u \in E$.
If $E ^{\ast} $ is strictly convex,
then $F_E $ is  a single-valued demi-continuous mapping.
We also obtain the following (see Pr\"{u}{\ss} \cite{Pru}):
\begin{proposition}
  A Banach Space $E$ is uniformly convex if and only if
   for each $R > 0 $ there exists a non-decreasing function $m _R : [0 ,\infty ) \to [0 ,\infty ) $
which satisfies $m _R (0) = 0 $, $m _R (\rho ) > 0 $ for any $\rho > 0 $, and 
\begin{align*}
\LA  u ^{\ast}_1 - u ^{\ast}_2,   u _1 - u _2  \RA _E
& \geq m _R (| u_1- u_2| _E ) | u_1- u_2| _E,\\
\hspace{7mm}
&\text{for any } u_1 , u _2 \in B_R
\text{ and } u ^{\ast}_1 \in F_E (u_1),~ u ^{\ast}_2 \in F_E (u_2),
\end{align*}
where $B_R : = \LD u \in E ;~ |u| _E  \leq R \RD $.
\label{Pro2.1} 
\end{proposition}

  Let $A$ be a (possibly) multivalued operator from $E$ into $2^{E^{\ast}}$(the power set of $E^{\ast}$) 
   and we often identify $A$ with its graph $G(A)$, more precisely, we write 
    $[u, u ^{\ast}] \in A$ if $ u \in D(A) := \{ ~\! v \in E~\! ; ~\! A v \neq \varnothing ~\! \}$ 
      and $ u ^{\ast} \in A u$. Operator $A$ is said to be monotone if 
\begin{equation*}
       \LA  u ^{\ast} _1 - u ^{\ast} _2 , u_1 -u_2  \RA _E \geq 0 \quad \forall [u_i, u ^{\ast} _i] \in A \ (i=1,2)
\end{equation*}
 and
 monotone operator $A$ is said to be maximal monotone if there is no monotone operator 
  which contains $A$ properly.
If $E$ and $E^{\ast}$ are reflexive and strictly convex, 
maximality of $A$  is equivalent to $ R( F_E + A) = E^{\ast}$ 
(see also Lemma \ref{lemma:maximal} below).

Let $\phi : E \to ( -\infty , + \infty ] $ be a 
proper (i.e., $\phi \not \equiv + \infty $) 
lower semi-continuous (l.s.c., for short) and convex functional.
The set $D(\phi ) := \{ u \in E; ~\phi (u) < + \infty  \} $
is called the effective domain of $\phi $ and
 the subdifferential operator $\partial \phi $ from $E$ into $2^{E^{\ast } }$ 
  is  defined by 
\begin{equation}
\partial \phi :
u \mapsto 
\{
\eta \in E ^{\ast };
~
\phi (v ) \geq \phi (u) + \LA \eta , v - u  \RA _E
~\forall v \in D(\phi )
\} .
\label{SubD} 
\end{equation}
  Then $\partial \phi $ becomes
a maximal monotone operator from $E$  into $2 ^{E ^{\ast }}$.
Here we  recall the following  fundamental feature of closedness of maximal monotone 
 operator 
(see Lemma 1.2 in Br\'{e}zis--Crandall--Pazy \cite{BCP}).
\begin{proposition}
\label{Pro2.2} 
Let $E$ be a real Banach space and $A:E \to 2^{E^{\ast}}$ be a maximal monotone operator. 
Assume that the sequences $\{u _n \} _{n \in \N } \subset E$ and $\{v _n \} _{n \in \N } \subset E^{\ast }$
satisfy $u _n \rightharpoonup u $ weakly in $E$,  $v _n \rightharpoonup v $ weakly in $E^{\ast}$,
$\LB u_ n  , v_ n  \RB \in A$, and either 
\begin{equation*}
\limsup _{n,m \to \infty } \LA  v _n - v_m , u_n -u_m  \RA _E \leq 0
\end{equation*}
   or
\begin{equation*}
\limsup _{n\to \infty  } \LA  v _n , u_n   \RA _E \leq \LA  v , u    \RA _E .
\end{equation*}
   Then $\LB u , v \RB \in A $ and $ \LA   v _n , u_n   \RA _E \to  \LA  v , u    \RA _E$ as $n \to \infty $.
\end{proposition}

Let $Z$ be another Banach space which is densely embedded in $E$ and $D(\phi ) \subset Z$.
Then we can consider the restriction of $\phi $ onto $Z$, denoted by $\phi _Z$
and regard $\phi _Z $ as a proper l.s.c. convex functional on $Z$.
To distinguish  between the two subdifferentials over $E $ and $Z$,
we write
\begin{align*}
     & \partial_E \phi (u) := \partial \phi (u) 
         = \{ \eta \in E ^{\ast } ; ~ \phi(v) \geq \phi(u) + \LA \eta, v - u \RA_E
	       ~ \forall v \in D(\phi ) \}, 
\\[2mm]
     & \partial _Z \phi _Z (u) := \{ \eta \in Z ^{\ast } ; 
         ~ \phi_Z (v) \geq \phi_Z (u) + \LA \eta, v- u \RA_Z
	       ~\forall v \in D(\phi _Z ) \}. 
\end{align*}
In general, we have $\partial _E \phi \subset \partial _Z \phi _Z$.
  More precisely it is shown that (see Proposition 2.1 in Akagi--Stefanelli \cite{AS_Cau})
\begin{equation}
\label{sub} 
   \begin{split}
      & D(\partial_E \phi ) = \{~\! u \in D(\partial_Z \phi_Z) ~\! ; ~\! 
                                  \partial_Z \phi_Z (u) \cap E^{\ast} \neq \varnothing ~\! \}, 
\\[2mm] 
      & \partial_E \phi (u) = \partial_Z \phi_Z (u ) \cap E^{\ast} 
          ~~~\forall u \in D(\partial_E \phi). 
   \end{split}
\end{equation}

We next define another functional derivative.
Let $\psi $ be a functional defined over $E$.
Then $\psi $ is said to be G\^{a}teaux differentiable at $u \in E$
if there exists $\xi \in E ^{\ast } $ such that
\begin{equation*}
\lim _{h\to 0 } \frac{\psi (u + he ) - \psi (u)}{h} 
=
\LA \xi , e \RA _E \hspace{5mm} \forall e \in E.  
\end{equation*}
We call $\xi $  the G\^{a}teaux derivative of $\psi $ at $u$
and 
write $\xi = d_E \psi (u)$. 
When there exists $ d_E \psi (u)$ for every $u \in E $,
$\psi $ is said to be G\^{a}teaux differentiable over $E$.
It is well known that
if $\psi $ is proper l.s.c. convex and G\^{a}teaux differentiable at $u$,
then G\^{a}teaux derivative $d _ E \psi (u)$ coincides with 
the subdifferential $\partial _E \psi (u )$
(i.e., $\partial _E \psi (u) = \{ d _ E \psi (u) \}$ holds).

For proper l.s.c. convex functional $\phi $ on $E$, we define
\begin{equation*}
\phi _\lambda (u)
:=
\inf _{v \in E} \LD \frac{ |u -v | ^2 _E}{2 \lambda } + \phi (v) \RD
= \frac{ | u - J_{\lambda } u| ^2_E}{2 \lambda } + \phi ( J_{\lambda } u )
\hspace{5mm} \lambda >0 ,~u \in E,  
\end{equation*}
where $J _{\lambda } : E\to E$ stands for the resolvent of $\partial _E \phi $ 
 , i.e., $J_\lambda u$ is the unique solution of 
\begin{equation*}
   F_E( J_\lambda u - u) + \lambda \partial _E \phi( J_\lambda u ) \ni 0.
\end{equation*}
This functional $\phi _{\lambda }$, called  the Moreau-Yosida regularization of $\phi $,
is convex, continuous, and G\^{a}teaux differentiable over $E$.
We can show that the G\^{a}teaux differential $d _E \phi _{\lambda }$
coincides with the
Yosida approximation $ (\partial_E  \phi)_\lambda $ of the subdifferential $\partial_E  \phi $, 
which is defined by
\begin{equation*}
  (\partial _E  \phi)_\lambda u := - \lambda^{-1} F_E(J_\lambda u - u).
\end{equation*} 
Moreover, we have $D(\phi_{\lambda }) = E  $ and for any $u \in  D(\phi) $ 
\begin{equation}
d _E \phi _{\lambda } (u) \in \partial _E \phi (J_{\lambda } u),~~
 \phi (J _{\lambda } u ) \leq \phi _{\lambda } (u) \leq \phi (u),~~
\lim_{\lambda \to +0 } \phi _{\lambda } (u) = \phi (u).
\label{Yosida} 
\end{equation}

We here present some useful tools,  the chain rule and 
	a formula concerning the integration by parts.
\begin{proposition}
\label{Pro2.3} 
(see Lemma 4.1 of Colli \cite{C})
Let $E$ be a strictly convex reflexive Banach space 
and $\phi : E \to ( -\infty , + \infty ]$ be a proper 
lower semi-continuous convex function.
Assume that $u \in W^{1,p} (0,T ;E )$ with $p \in (1, \infty )$,
 $\eta \in L^{p'} (0,T ;E^{\ast } )$,
 and $\eta (t) \in \partial \phi (u(t))$ a.e. $t \in (0, T )$.
Then $t \mapsto \phi (u (t))$ is absolutely continuous on $[0,T]$
and 
\begin{equation*}
\frac{d}{dt} \phi (u (t)) = \LA \eta (t) ,  u'(t) \RA _E
~~~\text{ for a.e. }  t \in (0,T). 
\end{equation*}
\end{proposition}


\begin{proposition}
\label{Pro2.4} 
(see Proposition 2.3 of Akagi--Stefanelli \cite{AS_Cau})
Let $m , p \in (1, \infty )$ and $E ,Z $ be reflexive Banach spaces such that
$Z \hookrightarrow E$ and $E ^{\ast} \hookrightarrow Z ^{\ast }$ are dense.
Moreover, let $u \in L^{m} (0,T  ; Z) \cap W^{1,p} (0, T ; E )$,
$\xi \in L^{p'}(0,T ; E ^{\ast})$, and its derivative of distribution 
$\xi ' $ belong to $  L^{m'} (0,T ; Z ^{\ast }) +  L^{p'} (0,T ; E ^{\ast }) $.
Provided that $t_1 , t_2 \in (0,T)$
are Lebesgue points of the function $t \mapsto \LA \xi (t) , u (t) \RA _{E}$,
then it holds that
\begin{equation*}
\LA \xi ' , u \RA _{L^{m} (t _1 , t_2  ; Z ) \cap  L^{p} (t _1 , t_2  ; E)}
=
\LA \xi (t _2) , u (t _2) \RA _{E} - \LA \xi (t _1) , u (t _1) \RA _{E} 
-
\LA \xi  , u ' \RA _{  L^{p} (t _1 , t_2  ; E)} .
\end{equation*}
\end{proposition}

\subsection{Assumptions}
In order to  formulate our results, we introduce a 
	suitable Banach space $V$ and its subspace $X$.
\begin{itemize}
\item[(A.0)] Both $V$ and its dual $V ^{\ast }$ are uniformly convex real Banach spaces 
               and $X$ is a subspace of $V$ such that 
                 $X$ and its dual $X ^{\ast }$ are real reflexive Banach spaces.
                 Furthermore the embeddings
\begin{equation*}
X \hookrightarrow V,~~~~~V^{\ast} \hookrightarrow X^{\ast }
\end{equation*}
hold with densely defined compact canonical injections.
\end{itemize}
We also assume the following growth conditions:
\begin{itemize}
\item[(A.1)]  Functional $\phi : V \to ( - \infty , + \infty ] $ is proper lower semi-continuous convex
and $\psi : V \to ( - \infty , + \infty )$ is convex and G\^{a}teaux differentiable over $V$.
Assume that
there exist a constant $C  >0 $ satisfying
\begin{align}
& | u | ^p _V \leq C ( \psi (u) + 1) ~~~\forall u \in V,
	\label{A1}  \\
& | d _V \psi (u)  | ^{p'} _{V ^{\ast }} \leq C( |u |^p _V + 1 ) ~~~\forall u \in V,
	\label{A2}  \\
& | u | ^m _X \leq  C ( \phi (u) + 1)  ~~~\forall u  \in D(\phi ),	
	\label{A3}  \\
& | \eta   | ^{m'} _{X ^{\ast }} \leq C ( |u |^m _X + 1) ~~~\forall \LB u  ,\eta \RB \in \partial _X \phi _X,
	\label{A4} 
\end{align}
 with some exponents  $p , m \in (1, \infty )$,
 where $p' := p /(p-1) $ and $m' := m / (m-1 )$ (H\"{o}lder conjugate exponents).
\end{itemize}

\begin{remark}
i) From \eqref{A1} and \eqref{A3},
 we assume that $\phi \geq 0$ and $\psi \geq 0$ without loss of generality henceforth
(replace $\phi $ and $\psi$ with $\phi +1 $ and $\psi +1$, respectively).

ii) Combining \eqref{A1}  with \eqref{A2}, we can immediately see that
\begin{equation}
|d _V \psi (u)| ^{ p' }_{V ^{\ast}} \leq C ( \psi (u) +1 )~~~\forall u \in V.
\label{A5} 
\end{equation}
Moreover,
since $\psi (u) \leq \psi (0) + \LA d _V \psi (u) , u \RA _V  $
holds by the definition of the subdifferential, \eqref{A2} leads to 
\begin{align}
&\psi (u) \leq C(|u | ^p _V +1)~~~~~~\forall u \in V,
\label{A6} \\
&|u | ^{p } _V \leq C(|d _V \psi (u) | ^{p'} _{V^{\ast }} +1)~~~~~~\forall u \in V.
\label{AW1} 
\end{align}
From \eqref{A3} and \eqref{A4}, we can derive analogous inequalities  for $\phi $: 
\begin{align}
& |\eta | ^{ m' }_{X ^{\ast}} \leq C( \phi (u) +1 )~~~~~~\forall \LB  u , \eta \RB  \in \partial _X \phi _X,
\label{A7} \\ 
& \phi (u) \leq C ( |u | ^m _X +1 )~~~~~~~\forall  u  \in D(\partial _X \phi _X),
\label{A8} \\
& |u | ^{m } _X \leq C (|\eta | ^{m'} _{X^{\ast }} +1)~~~~~~\forall \LB  u , \eta \RB  \in \partial _X \phi _X.
\label{AW2} 
\end{align}

\end{remark}

\subsection{Statements of Main Results} 
The following result is concerned with the solvability of 
	the elliptic regularization problem for (AP).
\begin{theorem}
\label{Th1}
Assume {\rm (A.0)} and {\rm (A.1)} with $m >  p$.
 Then for any $\varepsilon > 0 $ and $f \in L^{p'}(0,T; V^{\ast })$,
   the time-periodic problem
\begin{equation*}
\text{\rm (AP)}_{\varepsilon  }
\begin{cases}
~ -\varepsilon (d _V \psi (u ' _{\varepsilon } (t) )) ' +
	d _V  \psi (u ' _{\varepsilon } (t ))  + \partial _X \phi _X (u _{\varepsilon }  (t )) \\
~\hspace{30mm}
		+ \varepsilon F_V(u  _{\varepsilon } (t)) + \varepsilon d _V  \psi (u  _{\varepsilon } (t) )
		\ni  f(t )  ~~~~~~~ t \in (0,T ) ~~\text{ in } X ^{\ast} , \\
~u_{\varepsilon } (0) = u_{\varepsilon } (T) ,\\
~d _V \psi (u ' _{\varepsilon } (0) ) =  d _V \psi ( u '_{\varepsilon } (T) ) ,
\end{cases}
\end{equation*}
possesses at least one solution  $u_{\varepsilon}$ satisfying
\begin{equation}
\begin{split}
\label{Regu_T1}
&u _{\varepsilon } \in W ^{1,p} (0,T ;V) \cap L^m (0,T ;X ), \\
&d_V \psi (u _{\varepsilon } ) , d_V \psi (u   ' _{\varepsilon }  ) , F_V (u _{\varepsilon } ) 
			\in  L^{p'} (0,T ;V ^{\ast } ), \\
&\eta _{\varepsilon } \in L^{m'} (0,T ;X ^{\ast } ) , \\
&( d_V \psi (u  ' _{\varepsilon }  ) )'  \in  L^{p'} (0,T ;V ^{\ast } ) + L^{m'} (0,T ;X ^{\ast } ),
\end{split}
\end{equation}
and
\begin{equation*}
 \int_{0}^{T} \LA d _V \psi (u' _{\varepsilon } (t) ) , u' _{\varepsilon } (t) \RA _V dt
 \leq \int_{0}^{T} \LA f (t)  , u' _{\varepsilon } (t) \RA _V dt ,
\end{equation*}
where 
 $\eta _ {\varepsilon }$ is the section of $ \partial _X \phi _X (u_{\varepsilon } )$ 
  satisfying {\rm (AP)}$_{\varepsilon } $, 
i.e., $\eta _ {\varepsilon }$ satisfies $\eta _ {\varepsilon }(t) \in \partial _X \phi _X (u _{\varepsilon } (t))$
and 
$-\varepsilon (d _V \psi (u ' _{\varepsilon } (t) )) ' + d_V \psi (u  ' _{\varepsilon } (t))
+ \eta _{\varepsilon }(t) + \varepsilon F _V( u _{\varepsilon } (t))
+ \varepsilon d_V \psi (u  _{\varepsilon } (t))
		 =f(t )  $ for a.e. $t \in (0,T )$.
\end{theorem}

Via this approximation,
the solvability of original problem will be assured as follows:
\begin{theorem}
\label{Th2}
 Assume {\rm (A.0)} and {\rm (A.1)}.
  Then for every $f \in L^{p ' } (0, T ; V ^{\ast })$,
   {\rm (AP)} possesses at least one solution $u$ satisfying
\begin{equation}
\begin{split}
\label{Regu_T2}
  & u \in W ^{1,p} (0,T ;V) \cap L^\infty (0,T ;X ), 
\\
  & \phi(u(t)) \ \text{ is absolutely continuous on} \ [0,T], 
\\
  & d_V \psi (u'), \  \eta  \in  L^{p'} (0,T ;V ^{\ast } ), 
\end{split}
\end{equation}
 where $\eta $ is the section of $ \partial_V  \phi (u )$ satisfying {\rm (AP)}, 
   i.e., $\eta $ satisfies $\eta (t) \in \partial_V \phi  (u  (t))$
    and $d_V \psi (u' (t)) + \eta (t) = f(t )  $ for a.e. $t \in (0,T )$.
\end{theorem}
\begin{remark}
\label{Compare} 
i) In so far as  $A$ is given as a G\^{a}teaux differential, i.e., 
$A = d_V \psi $ in \eqref{Int03},
we see that our result can cover that of  Akagi--Stefanelli \cite{AS_Per}
by letting $V= X$ be Hilbert spaces and $p= 2$.

ii) In Akagi-Stefanelli \cite{AS_Cau} (Cauchy problem),
     they assume almost the same growth condition as (A.1)
      and allow that constants $C$ in \eqref{A3} and \eqref{A4} depend on $| u | _V$.
    In order to reveal a difference of  nature between Cauchy problem and the periodic problem,
we here check  how to derive  a priori estimate of solution to (AP).
 Testing (AP) by $u'$ and applying Proposition \ref{Pro2.3}, we get
\begin{equation*}
 \LA d _V \psi (u' (t)) , u' (t)  \RA _V + \frac{d}{dt} \phi (u(t)) = \LA f (t) , u' (t) \RA _V . 
\end{equation*}
Integrating over $[0,t]$ and using \eqref{A1}, we have for every $t \in [0,T]$
\begin{equation*}
\int_{0}^{t} | u' (s) | ^{p}_V  ds + \phi (u(t))
\leq C \LC  | f | ^{p'}_{L^{p'} (0,T ; V^{\ast})} + \phi (u(0)) + T
 \RC  
\end{equation*}
with some suitable constant $C > 0 $.
Hence for Cauchy problem, since $ u(0 ) $ is given data, 
one can derive estimates of $|u' | _{L^p (0,T ; V )} $ and $\sup _t \phi (u(t))$ immediately.
 For the periodic problem, however, 
since $ u(0 ) $ is unknown, 
 the boundedness of $|u' | _{L^p (0,T ; V )}$ can be assured
  but one can not get any information on $\sup _t \phi (u(t))$ 
   from the procedure above.
\end{remark}


\section{ Proofs of Main Results  } 
\subsection{Scheme of Proofs} 

 Our proofs consist of the following four steps:
\begin{itemize}
\item[\bf Step 1]
 We first deal with the case where $m >  p$  and show that 
\begin{equation*}
\text{(AP)}^{h} _{\lambda }
\begin{cases}
~ -\varepsilon (d _V \psi (u ' _{\lambda  } (t) )) ' + \varepsilon d_V \psi (u  _{\lambda  } (t))
		+ \partial _V \phi _{\lambda } (u _{\lambda } (t )) + \varepsilon F _V( u _{\lambda } (t))
	 \\
~\hspace{5cm}= f(t ) +h (t) ~~~~~~~ t \in (0,T ) ~~\text{ in } V ^{\ast} , \\
~u_{\lambda } (0) = u_{\lambda } (T) ,\\
~d _V \psi (u ' _{\lambda } (0) ) =  d _V \psi ( u '_{\lambda } (T) ) ,
\end{cases}
\end{equation*}
possesses a unique solution for every $f, h \in L^{p' } (0, T ; V ^{\ast})$ and $\lambda , \varepsilon >0 $,
 where $\phi _\lambda $ denotes the Moreau--Yosida regularization of $\phi$
  and $\cdot ~' $ stands for the time derivative.
 By letting $\lambda \to 0$, 
   we next  show the following result.
\begin{lemma}\label{lemma3.1} 
	Let $ 1< p < m $ and $f, h \in L^{p'}(0,T; V^\ast)$. Then for any $\varepsilon >0$,   
  \begin{equation*}
    \text{{\rm (AP)}}^{h} 
     \begin{cases}
       ~ -\varepsilon (d _V \psi (u ' _{h  } (t) )) ' + \varepsilon d_V \psi (u  _{h  }(t))
		+ \partial _X \phi _X (u _{h } (t )) + \varepsilon F _V( u _{h } (t))
 \\
          ~\hspace{5cm}\ni f(t ) +h (t) ~~~~~~~ t \in (0,T ) ~~\text{ in } X ^{\ast} ,
 \\
            ~u_h (0) = u_h (T) ,
 \\
              ~d _V \psi (u ' _h (0) ) =  d _V \psi ( u '_h (T) ) ,
      \end{cases}
   \end{equation*}
     has a unique solution satisfying
  \begin{equation}
    \begin{split}
     \label{Regu_S1}
       & u _h \in W ^{1,p} (0,T ;V) \cap L^m (0,T ;X ), 
 \\
       & d_V \psi (u _h ) , d_V \psi (u   ' _h  ) , F_V (u _h ) \in  L^{p'} (0,T ;V ^{\ast } ), 
 \\
       & \eta _h \in L^{m'} (0,T ;X ^{\ast } ), 
 \\
       & ( d_V \psi (u  ' _h  ) )'  \in  L^{p'} (0,T ;V ^{\ast } ) + L^{m'} (0,T ;X ^{\ast } ), 
 \\
    \end{split}
  \end{equation}
   where $\eta _h $ is the section of $\partial _X \phi _X (u_h )$ satisfying {\rm (AP)}$^h $,
    i.e., $\eta _h (t) \in \partial _X \phi _X (u _h (t))$ and  
      $-\varepsilon (d _V \psi (u ' _{h  } (t) )) ' + \varepsilon d_V \psi (u  _{h  } (t))
		+ \eta _h(t) + \varepsilon F _V( u _{h } (t)) =f(t ) +h (t) $ for a.e. $t \in (0,T )$. 
\end{lemma}
 Note that $  L^{p'} (0,T ;V ^{\ast } ) + L^{m'} (0,T ;X ^{\ast } ) = L^{m'} (0,T ;X ^{\ast } ) $ holds,
  since  we here assume that $p <m $.
   However, we still write as \eqref{Regu_S1} for the sake of consistency with the notation in \cite{AS_Cau}.

\item[\bf Step 2]
Let $u _h $ be the unique solution of (AP)$^h$.
Define an operator $\beta $ by 
\begin{equation}
\beta : h \mapsto u_h \mapsto - d _V \psi (u ' _ h ), 
\label{Scauder}
\end{equation}
namely, $\beta (h) := -d _V \psi ( u_h ' )$.
We apply  the Schauder--Tychonoff fixed point theorem to $\beta $ in
$\Xi := L^{p'} (0,T ; V ^{\ast })$ endowed with its weak topology.  
 Let $\bar{h}$ be the fixed point of $\beta$, then $u_{\bar{h}}$ gives a solution of 
 	(AP)$_{\varepsilon}$ satisfying properties given in Theorem \ref{Th1}.

\item[\bf Step 3]
  By establishing a priori estimates for solutions of (AP)$_\varepsilon$  
  	independent of the parameter $\varepsilon $ and discussing the convergence of solutions, we show the solvability of (AP) for $p <m $.

\item[\bf Step 4]
We deal with the case where $ m \leq p $.
 
\end{itemize}


\subsection{Step 1 (Existence of Solutions of Auxiliary Equations)}
 \quad Since our argument in this step is based on the procedure in Section 3 of Akagi--Stefanelli   
   \cite{AS_Cau}, one may find some duplications in this subsection.
  We first prepare the following result on some variational problem
    associated with (AP)$^h_{\lambda }$.

\begin{lemma}
\label{LEM3-1} 
 Let $p < m$ and put $\bar{p} := \max ~\! \{ p, 2 \}$, $\Gamma := L^{\bar{p}} (0 ,T ;V )$, 
   $\Gamma^{\ast} := L^{(\bar{p})'} (0 ,T ;V ^{\ast})$,
     $\LA \cdot , \cdot \RA _{\Gamma } := \DS \int ^T _ 0 \LA \cdot , \cdot \RA _V dt $,
      and let $h , f  \in \Gamma ^{\ast}$.
 Then for any $\varepsilon , \lambda >0$ the functional over $\Gamma $ defined by 
\begin{equation*}
  I_{\varepsilon, \lambda } (u) :=
   \begin{cases}
     \DS \int_{0}^{T} \LC \phi _{\lambda } (u (t)) + \varepsilon \psi (u ' (t)) + \varepsilon \psi (u  (t))
				  + \frac{\varepsilon }{2} |u (t)| ^2 _V - \LA f (t) +h (t) , u(t) \RA _V \RC dt
\\ 
      \hspace{12mm} \text{ if } u \in W^{1,p}(0,T; V), ~ u(0) = u(T), 
				~ \psi (u(\cdot)),  \psi (u'(\cdot)) \in L^1(0,T),
\\   
       + \infty \hspace{4mm}  \text{ otherwise, }
    \end{cases}
\end{equation*}
 admits a global minimizer $u_{\lambda } $ on $\Gamma $.
   Moreover, $u _{\lambda } \in  W^{1, \bar{p} }  (0 ,T ;V ) $ 
    is a unique solution to {\rm (AP)}$^{h} _{\lambda }$ satisfying
\begin{equation}
   \begin{split}
     &  u_{\lambda }(0) = u_{\lambda }(T), ~~ 
          d_V \psi (u'_{\lambda } (0) ) =d_V  \psi (u'_{\lambda } (T) ),
\\[2mm] 
     &  d_V  \psi (u'_{\lambda })  \in   W ^{1, (\bar{p} )' }   (0,T ; V^{\ast } ) , 
          ~ d_V  \psi (u_{\lambda }) ,
	        ~\partial_V\phi_{\lambda } (u_{\lambda }), ~F_V(u_{\lambda }) 
               \in  L^{(\bar{p})'} (0,T ; V^{\ast }) .
   \end{split}
\label{Reg3-1} 
\end{equation}
\end{lemma}
\begin{proof}
It is easy to see that $I _{\varepsilon , \lambda }$ is proper lower semi-continuous and strictly convex 
on $\Gamma$.
By assumptions \eqref{A1} and \eqref{A3},
$I _{\varepsilon , \lambda }$ is coercive and bounded from below.
Then the standard argument guarantees the existence of global minimizer 
 $u_\lambda$  of $I _{\varepsilon , \lambda }$.

 Divide $I _{\varepsilon , \lambda }$ 
   into  the following five parts:
\begin{align*}
 I^1_{\varepsilon, \lambda } (u)
   &:=
	\begin{cases}
    	~ \DS \varepsilon  \int_{0}^{T} \psi (u' (t)) dt ~~
    	    & \text{ if } u \in 
    	        D(I^1_{\varepsilon, \lambda}) := \{  u \in W^{1, \bar{p} }(0,T ; V) ; u(0) =u(T) \}, 
    \\
	    ~ + \infty ~~ 
	        & \text{ otherwise},
	\end{cases}
\\[2mm]
 I^2_{\varepsilon, \lambda } (u)
  & := \int_{0}^{T} \phi_{\lambda }  (u (t)) dt, 
 \\[2mm]
 I^3_{\varepsilon, \lambda } (u)
   &:=
	\begin{cases}
    	~ \DS \varepsilon  \int_{0}^{T} \psi   (u (t)) dt ~~
    	    & \text{ if } \  \psi(u (\cdot )) \in L^1(0,T), 
  \\
    	~ + \infty ~~
    	    & \text{ otherwise},
	\end{cases}
\end{align*}
\begin{align*}
 I^4_{\varepsilon, \lambda } (u)
   &:= \frac{\varepsilon }{2}  \int_{0}^{T} |u (t) |^2_V dt, 
\\[2mm]
 I^5 _{\varepsilon ,\lambda } (u)
   &:=  \int_{0}^{T} \LA f (t) + h (t) , u(t) \RA _V dt. 
\end{align*}  
We first show that $\partial_{\Gamma} I ^1_{\varepsilon, \lambda }$ coincides with the operator
 $A :\Gamma \to \Gamma^{\ast }$ defined by  
\begin{equation*}
A u (t) := - \LC \varepsilon d _V \psi (u ' (t ) ) \RC ' 
\end{equation*}
with domain
\begin{equation*}
D(A) := \LD  u \in D(I^1_{\varepsilon, \lambda } ) ;~~
	d_V \psi (u' (\cdot )) \in  W^{1, (\bar{p})' }   (0,T ; V^{\ast } ),~~d _V \psi (u' (0 ))
	   = d _V \psi (u' (T ))\RD. 
\end{equation*}
Since $d_V \psi = \partial_V \psi$,
 we have for every $u \in D(A)$ and $v \in D(I^1_{\varepsilon, \lambda })$
\begin{align*} 
 I^1_{\varepsilon, \lambda } (v) - I^1_{\varepsilon, \lambda } (u) 
  & \geq  \varepsilon \int_{0}^{T} \LA d_V \psi (u ' (t) ), v ' (t) - u'(t) \RA _V dt
\\ 
  & = - \varepsilon \int_{0}^{T} \LA  ( d_V \psi (u ' (t) ) )' , v  (t) - u(t) \RA _V dt
\\ 
   & = \LA A u , v - u \RA_{\Gamma}, 
\end{align*}
which implies $A \subset \partial_{\Gamma} I^1 _{\varepsilon, \lambda }$.
 To prove the inverse inclusion,
   we set $\Lambda :=  W^{1,\bar{p}} (0,T ; V)$ and define two functionals 
    $L ,K : \Lambda \to [0, + \infty ]$ by 
\begin{equation}
L(u) := \varepsilon  \int_{0}^{T} \psi (u ' (t) ) dt,~~
K (u):=
\begin{cases}
~~0 ~&~\text{ if }u(0 ) =u(T),\\
+\infty ~&~\text{ otherwise}.
\end{cases}
\label{label-001} 
\end{equation}
Remark that the restriction of $I ^1 _{\varepsilon , \lambda } $ onto $\Lambda $,
denoted by $I ^1 _{\Lambda  } $ henceforth,
coincides with the sum of $L$ and $K$.
Fix $h \in \R $ and $ u , e \in \Lambda $ arbitrarily.
By the definition of subdifferential, 
\begin{equation*}
h \LA d_V \psi (u ' (t) ) , e ' (t)  \RA _V
\leq
\psi (u ' (t) + h e ' (t)  )  - \psi (u ' (t) )  
\leq 
h \LA d_V \psi (u ' (t) + h e '(t )) , e ' (t)  \RA _V 
\end{equation*}
holds for a.e. $t \in (0, T)$.
Thanks to \eqref{A2}, we can apply Lebesgue's dominated convergence theorem
 and we obtain 
 \begin{align*}
\frac{L(u + h e) -L(u)  }{h }
	& = \varepsilon \int_{0}^{T} \frac{\psi (u'(t) + h e'(t)) -\psi (u'(t))}{h} dt    \\
	& \to \int_{0}^{T} \LA \varepsilon d _V \psi (u'(t)) , e ' (t) \RA _V dt  
	    \ =  \LA  d_\Lambda L(u), e  \RA _\Lambda
	       \quad \text{as} \ \ h \to 0. 
\end{align*}
Hence the functional $L$ is G\^{a}teaux differentiable over $\Lambda $.
On the other hand,
since  $K $ is a proper indicator  function of the closed linear subset $\{ u \in \Lambda ; u(0)=u(T) \} $ of  $\Lambda $,
the subdifferential  of $K$ is well defined and 
$\LA v , e \RA _\Lambda = 0 $ holds for every $\LB u , v  \RB \in \partial _{\Lambda }K $ and 
$e \in \Lambda $ s.t. $e (0) = e(T)$. 
Since $D(d_\Lambda L) = \Lambda $ and $D(\partial _{\Lambda }K) =D(K)$, 
 we can derive
$\partial_{\Lambda } I^1 _{\Lambda } = \partial  _{\Lambda } (L +K) = d _{\Lambda } L + \partial  _{\Lambda } K $
and 
\begin{align*}
 D(\partial_\Lambda  I^1_{\Lambda }) 
	&= D(d _{\Lambda } L) \cap D(\partial  _{\Lambda} K ) \\
	&= \{ u \in \Lambda ;~  u (0) = u(T)\}
\end{align*}
(see, e.g., Theorem 2.10 in Barbu \cite{Bar2}).
Let $\LB u , v  \RB \in \partial _{\Gamma  } I^1_{\varepsilon , \lambda  }  $ and $e \in \Lambda $ s.t. $e (0) = e (T)$.
By the general relationship
$\partial _{\Gamma  } I^1_{\varepsilon , \lambda  } \subset  \partial _{\Lambda   } I^1_{\Lambda }$,
we have
\begin{equation*}
\LA  v  , e \RA _{\Lambda } =  \LA  d _\Lambda L (u)   , e \RA _{\Lambda }
= \int_{0}^{T}   \LA \varepsilon d _V \psi (u '(t) ) , e '(t) \RA _V dt.
\end{equation*}
In addition,
since 
$\DS \LA  v  , e \RA _{\Lambda }  = \DS \LA  v  , e \RA _{\Gamma } = \int_{0}^{T} \LA v(t) , e(t) \RA _V dt $
is
verified by $v \in \Gamma ^{\ast } \hookrightarrow \Lambda ^{\ast }$  and $e \in \Lambda \hookrightarrow \Gamma $,
\begin{equation*}
 \int_{0}^{T} \LA v(t) , e(t) \RA _V dt
= \int_{0}^{T}   \LA \varepsilon d _V \psi (u '(t) ) , e '(t) \RA _V dt
\end{equation*}
holds for every  $e \in \Lambda $ with $e (0) = e (T)$.   
 Hence by testing this by $e \in C_0^\infty((0,T);V)$, which is dense in $\Gamma$, 
	 we conclude that 
        $v= - ( \varepsilon d_V \psi (u'(\cdot )) )'  
          \in \Gamma^{\ast} =  L^{(\bar{p})'} (0,T ; V^{\ast })$.
Combining this with the integration by parts,
we obtain 
\begin{equation*}
\LA  d _V \psi (u '(T) ) , e (T)  \RA _V - \LA  d _V \psi (u '(0) ) , e (0)  \RA _V 
=
\LA  d _V \psi (u '(T) )  -  d _V \psi (u '(0) )  , e (0)  \RA _V 
=0 . 
\end{equation*}
Therefore $ d _V \psi (u '(T) )  =  d _V \psi (u '(0) ) $ in $V ^{\ast }$ by the arbitrariness of $e(0)$
and then $ \partial _{\Gamma }I^1 _{\varepsilon , \lambda } \subset A $ is shown.

We now check the remainders $ I^i_{\varepsilon, \lambda}(\cdot) \ (i=2,3,4,5)$. 
Obviously, these are proper lower semi-continuous convex functional defined on $\Gamma $.
We first get
\begin{equation*}
I^2 _{\varepsilon  , \lambda } (u)
 \leq \int_{0}^{T} \LC \frac{| u(t) - v | ^2 _{V}}{2\lambda } + \phi (v)  \RC dt 
 ~~\forall u \in  \Gamma \text{ and  }
 \forall  v \in D(\phi )
\end{equation*}
from the definition of the Moreau--Yosida regularization
and the definition of $\partial \psi$ and \eqref{A2} yield
\begin{align*}
I^3 _{\varepsilon  , \lambda } (u)
&  \leq \varepsilon T \psi (v)
+\varepsilon 
\int_{0}^{T} 
|d _V \psi (u (t)) |_{V^{\ast }} |u (t ) - v | _V dt \\
& \leq \varepsilon T \psi (v)
 +\varepsilon C \LC |u(t)| ^p _V + |v | ^p _V +1  \RC 
 ~~\forall u \in \Gamma \text{ and  }
 \forall v \in D(\psi ),
\end{align*}
where $C$ is some suitable constant independent of $\varepsilon $ and $\lambda $.
Hence 
$D(I^3 _{\varepsilon , \lambda })  = \Gamma $ and $I^4_{\varepsilon, \lambda}, \  
  I^4_{\varepsilon, \lambda}$ are well defined on $\Gamma$. 
   Moreover,  by the standard argument (see Appendice I in \cite{PU}), 
   we have
   \begin{align*}
        & \partial _{\Gamma}  I^2_{\varepsilon, \lambda }(u) 
            = \partial_V \phi_{\lambda } (u (t)),\hspace{5mm} 
         \partial_{\Gamma}  I^3_{\varepsilon, \lambda}(u) 
            = \varepsilon d_V  \psi  (u (t)), 
    \\
        & \partial_{\Gamma}  I^4_{\varepsilon , \lambda }(u) 
            = \varepsilon F_ V (u (t)), \hspace{5mm} 
         \partial_{\Gamma} I^5_{\varepsilon, \lambda }(u) 
            = - f(t) - h(t),
   \end{align*} 
 for any $u \in \Gamma $ and a.e. $t \in (0,T)$.
 Hence we immediately have
\begin{align*}
\partial _{\Gamma} I_{\varepsilon , \lambda } 
	 &=\partial _{\Gamma}  ( I ^1 _{\varepsilon , \lambda } +I ^2 _{\varepsilon , \lambda } 
					+I ^3 _{\varepsilon , \lambda } +I ^4 _{\varepsilon , \lambda }  +I ^5 _{\varepsilon , \lambda } ) \\
	&= \partial  _{\Gamma}  I ^1 _{\varepsilon , \lambda } +\partial _{\Gamma}  I ^2 _{\varepsilon , \lambda } 
					+\partial _{\Gamma}  I ^3 _{\varepsilon , \lambda } +\partial _{\Gamma}  I ^4 _{\varepsilon , \lambda }  
			+\partial _{\Gamma}  I ^5 _{\varepsilon , \lambda } .
\end{align*}
Therefore, since the global minimizer $u _\lambda $ satisfies 
$0 = \partial _{\Gamma} I _{ \varepsilon , \lambda  } (u_{\lambda })$,
we  conclude that $u _\lambda $ is the unique solution to (AP)$^{h} _{\lambda }$.
\end{proof}

\begin{remark}
If  one tries to apply the WED approach to $I^1 _{\varepsilon , \lambda }$ by the same argument as 
that in \cite{AS_Cau}, 
 one gets a solution with condition
 $d_ V \psi (u ' _{\lambda }(0)) = e ^{- T / \varepsilon } d_ V \psi  ( u '_{\lambda }( T ))$.
Taking the limit as $\varepsilon ,\lambda  \to +0 $,
 one is forced to face with the 
 over-determined problem with $d_ V \psi (u ' (0)) =  d_ V \psi  ( u '( T )) = 0$.
\end{remark}

We next take the limit  $\lambda \to 0$ in (AP)$^h_{\lambda }$.
 In this procedure, we note that if we multiply  
   the equation of (AP)$^h _{\lambda }$ by $u '_{\lambda }$  as in \cite{AS_Cau},
     we get for any $t_1 $ and $t_2$,
\begin{equation*}
\LB \varepsilon  \psi ^{\ast } (d_V \psi (u' (t))) 
+ \varepsilon \psi (u _{\lambda } (t))
 +\phi _{\lambda } (u_{\lambda } (t))
+ \frac{\varepsilon }{2} |u_{\lambda } (t)| ^2 _{V}   \RB ^{t =t _ 2} _{t =t _1}
 \leq \int_{t_1}^{t_2} \LA f(t)+ h(t) ,u' _{\lambda }(t) \RA _V dt  ,
\end{equation*}
where $\psi ^{\ast }$ is Legendre--Fenchel transform of $\psi $
(precise treatment of the first term of L.H.S. will be given in \eqref{St3-02} and \eqref{LF} below).
However, this inequality  seems to be not useful for the periodic problem,  
since $u _{\lambda } (0) $ is still  an unknown value. 
 So we here adopt a different manner to prove Lemma \ref{lemma3.1}.
 \begin{proof}[\sc Proof of Lemma \ref{lemma3.1}] 
 	 Let $ h, \ f \in L^{p'}(0,T;V^\ast)$. Since 
 		$ L^{p'}(0,T;V^\ast) \subset L^{(\bar{p})'}(0,T;V^\ast)$, 
 		  we note that Lemma \ref{LEM3-1} assures the existence of the unique solution $u_{\lambda }$ 
 		    of (AP)$^h_{\lambda }$ satisfying $u_{\lambda} \in W^{1,\bar{p}}(0,T;V) \subset W^{1,p}(0,T;V)$. 
 We first multiply
 the equation of (AP)$^h_{\lambda }$ by $u_{\lambda }$. 
Then using the integration by parts 
and the periodicity of $u_{\lambda }$, we have 
\begin{equation}
\label{St1-01} 
\begin{split}
 & \varepsilon \int_{0}^{T} \LA d_V \psi (u' _{\lambda } (t))  , u '_{\lambda } (t) \RA _V dt
	 +\varepsilon \int_{0}^{T} \LA d_V \psi (u _{\lambda } (t)) , u_{\lambda } (t) \RA _V dt \\
		&+  \int_{0}^{T} \LA \partial _V \phi _{\lambda } (u _{\lambda } (t)) , u_{\lambda } (t) \RA _V dt
		+  \varepsilon \int_{0}^{T} |  u_{\lambda } (t) |^2 _V dt
		=   \int_{0}^{T} \LA  f(t) + h(t) , u_{\lambda } (t) \RA _V dt .
\end{split}
\end{equation}
 The definition of subdifferential and \eqref{A1} yield
\begin{equation}\label{est:below:psi:v}
\int_{0}^{T}  \LA d _V \psi (v(t)) , v  (t) \RA _V dt
   \geq \int_{0}^{T} \LC \psi (v (t) ) - \psi (0) \RC dt
      \geq c_1 \LC |v | ^p _{L^p (0,T ; V )} - T \RC, 
\end{equation}
and hence by \eqref{est:below:psi:v} with $v=u'_\lambda$ and $v=u_\lambda$, we get 
 \begin{equation}
\label{St1-02} 
\begin{split}
 & \varepsilon c _1 \int_{0}^{T} \LC |u'_{ \lambda} (t)| ^p _V +  |u_{ \lambda} (t)| ^p _V  \RC dt 
		+  \int_{0}^{T} \LA \partial _V \phi _{\lambda } (u _{\lambda } (t)) , u_{\lambda } (t) \RA _V dt 
		+ \varepsilon \int_{0}^{T} |  u_{\lambda } (t) |^2 _V dt \\
		&\leq 
		\int_{0}^{T} \LA  f(t) + h(t) , u_{\lambda } (t) \RA _V dt 
		+ 2  \varepsilon  C_1 T.
\end{split}
\end{equation}
Here $c_1 $ and $C _1$ denote positive general constants independent of the parameters $\lambda $ and $ \varepsilon $. 
 From this and Young's inequality, we can derive
 \begin{equation}
\label{St1-020} 
\begin{split}
 & \varepsilon c _1 \int_{0}^{T} \LC |u'_{ \lambda} (t)| ^p _V +  |u_{ \lambda} (t)| ^p _V  \RC dt 
		+  \int_{0}^{T} \LA \partial _V \phi _{\lambda } (u _{\lambda } (t)) , u_{\lambda } (t) \RA _V dt 
		+ \varepsilon \int_{0}^{T} |  u_{\lambda } (t) |^2 _V dt \\
		&
		\leq C_{\varepsilon } \int_{0}^{T} |   f(t) + h(t) |^{p'} _{V ^{\ast }} dt 
		+ 2  \varepsilon  C_1 T ,
\end{split}
\end{equation}
where $C_{ \varepsilon } $ is a general constant depending only on $\varepsilon $.
Then we obtain the following estimate independent of $\lambda $:
\begin{equation}\label{est:energy:final}
  \int_{0}^{T} \LA \partial_V \phi_{\lambda } (u_{\lambda } (t)), u_{\lambda } (t) \RA_V dt 
    + \varepsilon \int_{0}^{T} \LC |u'_{\lambda } (t)|^p_V +  |u_{\lambda } (t)|^p_V
       +  |  u_{\lambda } (t) |^2_V \RC dt
         \leq C_{\varepsilon } .
\end{equation}
  Hence for any fixed $v \in  D(\phi) \subset X $,
	 the definition of subdifferential and the fact that 
        $\phi _{\lambda } (v) \leq \phi (v)$ yield 
\begin{align*}
0 &\leq \int _{0}^{T} \phi _{\lambda } (u _{\lambda } (t)) dt
\leq \int _{0}^{T} \LA  \partial _V \phi _{\lambda } (u _{\lambda } (t)) , u_{\lambda } (t) -v \RA _V dt 
	 + T \phi (v)  \\
&\leq C_{\varepsilon } +|v |_X 
 \int _{0}^{T} |  \partial _V \phi _{\lambda } (u _{\lambda } (t)) | _{X^{\ast}}  dt .
\end{align*}
Then since 
$\partial  _V \phi ( J _{\lambda } u _{\lambda }) 
\subset \partial _X \phi _X (J_{\lambda } u_{\lambda })$ 
and 
$\phi (J _{\lambda } u_{\lambda }) \leq  \phi _{\lambda }  (u_{\lambda }) \leq \phi (u_{\lambda })$
 (recall  \eqref{Yosida}),
we have from \eqref{A3} and \eqref{A4} 
\begin{align*}
 \int _{0}^{T} \phi _{\lambda } (u _{\lambda } (t)) dt
&\leq C_{\varepsilon } +C |v |_X 
 \int _{0}^{T} ( \phi (J_{\lambda } u _{\lambda } ) +1 ) ^{1/ m'} dt \\
 &\leq C_{\varepsilon } +C T ^{1/m} |v |_X 
 \LC \int _{0}^{T} ( \phi _{\lambda } ( u _{\lambda } ) +1 ) dt \RC  ^{1/ m'} ,
\end{align*}
whence follows
\begin{equation}
\label{St1-03} 
0 \leq \int _{0}^{T} \phi _{\lambda } (u _{\lambda } (t)) dt
 \leq C_{\varepsilon }.
\end{equation}  
Hence, again by \eqref{Yosida}, \eqref{A3} and \eqref{A4}, we obtain  
\begin{equation*}
\int_{0}^{T} \LC
|J _{\lambda } u_ {\lambda } (t)| ^m _{X}
+
|\partial _V \phi _{\lambda } (u_ {\lambda } (t) )| ^{ m '} _{X^{\ast }}
 \RC dt 
 \leq C_{\varepsilon },
\end{equation*}
  which together with \eqref{est:energy:final} yields 
\begin{equation}
\label{St1-04} 
\begin{split}
| \partial _V \phi _{\lambda } (u _{\lambda })  | _{L ^{m '} (0, T ; X ^{\ast})}
&+
|u _{\lambda } | _{W ^{1 , p} (0, T ; V)} \\
&+
| u _{\lambda }  | _{L ^{2} (0, T ; V)} 
+
| J_{\lambda } u _{\lambda }  | _{L ^{m } (0, T ; X )}
 \leq C_{\varepsilon }.
\end{split}
\end{equation}
Therefore  \eqref{A2} leads to
\begin{equation}
\label{St1-05} 
\int_{0}^{T} | d_V  \psi (u_{\lambda } (t))| ^{p'} _{V^{\ast }} dt 
+
\int_{0}^{T} | d_V  \psi (u ' _{\lambda } (t))| ^{p'} _{V^{\ast }} dt 
\leq C_{\varepsilon }
\end{equation}
and  \eqref{A6} implies 
\begin{equation}
\label{St1-06} 
\int_{0}^{T} \psi (u_{\lambda } (t)) dt 
+
\int_{0}^{T} \psi (u ' _{\lambda } (t)) dt 
\leq C_{\varepsilon }.
\end{equation}
 Returning to (AP)$^h _{\lambda }$, we get
\begin{equation}\label{St1-07} 
   \begin{split}
      & |\varepsilon (d_V \psi ( u'_{\lambda } ))' |_{L^{p'} (0, T ; V^{\ast}) + L^{m'} (0, T ; X^{\ast})} 
   \\
      &  \hspace{5mm} \leq |\varepsilon  d_V \psi ( u_{\lambda} ) |_{L^{p'} (0, T ; V^{\ast}) }
                        + |\partial_V  \phi_{\lambda } ( u_{\lambda}) |_{ L^{m'} (0, T ; X^{\ast}) } 
   \\
      & \hspace{15mm} + | \varepsilon  F_V (u_{\lambda}) |_{L^{p'} (0, T ; V^{\ast})} 
                        + | f+ h |_{L^{p'} (0, T ; V^{\ast})} 
                          \leq  C_{\varepsilon}.
   \end{split}
\end{equation}
Here we used the fact that $W^{1,p}(0,T;V) \subset L^\infty(0,T;V)$ 
 	            and $| F_V(u_\lambda)|_{V^\ast} = |u_\lambda|_V$.

     Then by \eqref{St1-04}, \eqref{St1-05} and \eqref{St1-07}, we can extract 
 a subsequence  $\{ \lambda _n \} _{n \in \N }$
such that $\lambda _n \to 0 $ as $n \to \infty $  and 
\begin{equation}\label{conv:u:lambdan} 
   \begin{split} 
      u_{\lambda_n} \rightharpoonup \exists u ~~~~ & \text{ weakly in } W^{1,p}(0,T ; V ), 
 \\
      J_{\lambda_n} u_{\lambda_n} \rightharpoonup \exists v ~~~~ & \text{ weakly in } L^{m}(0,T ; X ), 
 \\
      F_V (u_{\lambda_n  }) \rightharpoonup \exists w ~~~ & \text{ weakly in } L^{p'}(0,T ; V^{\ast} ),
\\
      \partial_V \phi_{\lambda_n} (u_{\lambda_n}) \rightharpoonup \exists \eta
					                                  ~~~~ & \text{ weakly in } L^{m'} (0,T ; X^{\ast} ), 
\\
       d_V \psi  (u_{\lambda _n}) \rightharpoonup \exists a
	                                   				  ~~~~ & \text{ weakly in } L^{p'} (0,T ; V^{\ast} ), 
\\
       d_V \psi  (u'_{\lambda _n}) \rightharpoonup \exists \xi
					                                  ~~~~ & \text{ weakly in } L^{p'} (0,T ; V^{\ast} ), 
\\
       ( d_V \psi  (u'_{\lambda _n}) )' \rightharpoonup  \xi' 
					~~~ & \text{ weakly in } L^{p'} (0,T; V^{\ast} )+ L^{m'} (0,T ; X^{\ast} ),
   \end{split}
\end{equation}
where $\xi ' $ stands for the derivative of $\xi $ in the sense of distributions.
Taking the limit in (AP)$^h _{\lambda _{n}}$, we have
\begin{equation*}
  - \varepsilon \xi ' + \varepsilon  a + \eta +\varepsilon w = f + h 
     \quad \text{ in } \quad  L^{p'} (0,T ; V ^{\ast} )+ L^{m'} (0,T ; X ^{\ast} ).
\end{equation*}
 Now we are going to show that $u$ gives a solution of (AP)$^h$. 
  	We first note that $\{ u_ \lambda \} _{\lambda >0}$ forms an equi-continuous family in $C ([0, T ] ; V )$, 
  	  which is assured by 
          $| u' _{\lambda } | _{L^{p} (0, T ; V ) } \leq C_{\varepsilon } $.
  Moreover,  since $u_{\lambda } \in C([0,T];V)$, there exists 
    $t_0 \in [0,T)$ such that
\begin{equation*}
 T^{1/p} |u_\lambda (t_0)|_V = 
    \ T ^{1/p} \min _{t \in [0,T]} | u _{\lambda }  (t ) | _ V 
       \leq 
          | u _{\lambda } | _{L^{p} (0, T ; V ) } \leq C_{\varepsilon }.
\end{equation*}
 Then for any $t \in [t_0, t_0 + T],$ we get 
\begin{equation*}
  |u_\lambda(t)|_V \leq |u_\lambda(t_0)|_V + \int_{t_0}^t |u_\lambda' ( s)|_V ~\! ds 
                     \leq |u_\lambda(t_0)|_V + T^{1/p'}  |u_\lambda'|_{L^p(0,T;V)}
                       \leq C_\varepsilon,
\end{equation*} 	
 which implies 	
\begin{equation}
\label{St1-08} 
\sup _{t \in [0,T]} | u _{\lambda }  (t ) | _ V \leq C_{\varepsilon } .
\end{equation}
Thanks to the general property $ | J_{\lambda } u _{\lambda }| _V \leq C_1 (|u _{\lambda }|_V +1) $,
we also get 
\begin{equation}
\label{St1-09} 
\sup _{t \in [0,T]} | J_{\lambda } u _{\lambda }  (t ) | _ V \leq C_{\varepsilon } .
\end{equation}
We here recall that $J _{\lambda } u_\lambda$ is defined by 
 the unique solution of $ F _V (J _{\lambda } u _{\lambda } - u_{\lambda } ) \in - \lambda 
\partial _V \phi (J _{\lambda }  u_{\lambda })$.
Then from the monotonicity of $\partial _V \phi $, 
\begin{equation*}
\LA F _V (J _{\lambda } u _{\lambda } (t + h ) - u_{\lambda } (t+h ) )
- F _V (J _{\lambda } u _{\lambda } (t  ) - u_{\lambda } (t ) ),
J _{\lambda }  u _{\lambda } (t + h ) - J _{\lambda }  u_{\lambda } (t ) 
  \RA _V \leq 0
\end{equation*}
holds for any $h \in \R$.
On the other hand, \eqref{St1-08} and \eqref{St1-09} yield 
\begin{align*}
&\LA F _V (J _{\lambda } u _{\lambda } (t + h ) - u_{\lambda } (t+h ) )
- F _V (J _{\lambda } u _{\lambda } (t  ) - u_{\lambda } (t ) ),
 -u _{\lambda }(t+h )  +u_{\lambda } (t ) \RA _V \\
&\leq 
C_{\varepsilon } |u_{\lambda } (t+h )  -u_{\lambda } (t ) | _V .
\end{align*}
Then by adding two inequalities above, we obtain by Proposition \ref{Pro2.1} 
\begin{align*}
   &  m_{C_{\varepsilon }} ( |\LC J_{\lambda} u_{\lambda}(t+h) - u_{\lambda}(t+h) \RC
		 - \LC J_{\lambda} u_{\lambda}(t) - u_{\lambda}(t) \RC |_V ) 
\\
   & ~~~~~~~~\times |\LC J_{\lambda} u_{\lambda} (t+h) - u_{\lambda}(t+h) \RC
	     -\LC J_{\lambda} u_{\lambda}(t) - u_{\lambda}(t) \RC |_V  \ 
	         \leq C_{\varepsilon} |u_\lambda(t+h) - u_\lambda(t) |_V,
\end{align*}
which implies  that $ \{ J _{\lambda } u _{\lambda } - u_{\lambda }  \} _{\lambda > 0 }$
  forms  an equi-continuous family in $C([0,T]; V)$. 
   Hence so does $ \{ J _{\lambda } u _{\lambda }  \} _{\lambda > 0 }$.
Therefore 
Theorem 3 of Simon \cite{Simon} assures the relative compactness of
$ \{ J _{\lambda } u _{\lambda }  \} _{\lambda > 0 }$ in  $C([0,T]; V)$,
i.e., 
\begin{equation*}
J _{\lambda _n } u _{\lambda _n} \to
v ~~~\text{ strongly in } C ([0,T] ; V )
\end{equation*}
(recall that $v$ is the weak limit $\{ J _{\lambda _n } u _{\lambda _n}  \} _{n \in \N } $
in $L^{m} (0,T; X)$).
Since $J _{\lambda } u _{\lambda } (0) = J _{\lambda } u _{\lambda } (T)$,
 the limit $v$ is also time-periodic.
   Furthermore,  by the definition of $\phi _{\lambda }$ and 
      the fact that $\phi \geq 0$, we have
\begin{equation*}
\int_{0}^{T} | u_{\lambda } (t) - J_{\lambda } u_{\lambda } (t) | ^2_V dt 
  \leq 2 \lambda \int_{0}^{T} \phi _{\lambda } (u _{\lambda } (t)) dt 
   \leq 
     2 \lambda C_{\varepsilon }, 
\end{equation*}
  which implies that $\{ u_{\lambda } - J _{\lambda } u_{\lambda } \} _{\lambda >0 }$
   converges to $0$ strongly  in $L^2 (0,T ;V)$ and then $u = v$.
Using  the strong convergence of $\{  J _{\lambda _n } u_{\lambda _n } \} _{n \in \N }$
and uniform boundedness \eqref{St1-08} and \eqref{St1-09},
we  also have 
\begin{equation}\label{conv:u:lambdan:Lr}
  u _{\lambda _n } \to u 
      ~~~\text{ strongly in } L^r (0,T ; V )~~\forall r \in [1, \infty )
\end{equation}
 and then by the demiclosedness of $F_V$ and  $d_V \psi$, we get 
\begin{equation}\label{identity:F:dVpsi} 
  w = F _V (u), \quad a = d _V \psi (u).
\end{equation}
Here we can extract a subsequence of $\{  u_{\lambda _n } \} _{n \in \N }$
(we omit relabeling) such that
\begin{equation}
\label{St1-10}
u _{\lambda _n } (t) \to u(t) 
 ~~~\text{ strongly in }  V ~~\text{ for a.e. } t \in (0,T ).
\end{equation}
Since 
 \eqref{St1-05}, \eqref{St1-07},  and compact embedding $V ^{\ast } \hookrightarrow X ^{\ast } $ holds,
 Theorem 3 of \cite{Simon} is also applicable to the sequence $\{ d _V \psi (u '_{\lambda _n }) \} _{n \in \N }$.
Readily, 
\begin{equation}
\label{St1-11} 
 d _V \psi (u '_{\lambda _n }) \to \xi 
 ~~~\text{ strongly in } C ([0,T] ; X ^{\ast } ) .
\end{equation}
Remark that $\xi (0 ) = \xi (T)$ also can be deduced
 from the periodicity of $d _V \psi (u '_{\lambda _n })$.

In order to complete this step,
we have to check $\xi = d_V \psi ( u' )$ and $\eta \in \partial _ X \phi _X ( u)$.
We here define  a subset $\Upsilon $ of $(0,T)$ by 
\begin{equation*}
\Upsilon 
:=
\LD 
\begin{array}{l|l}
~~ &
 t \text{ is a Lebesgue point of } t \mapsto \LA \xi (t) ,  u (t) \RA _V \text{ and }  \\
 t \in (0, T ) &
  \{ \lambda _n \} \text{ has a subsequence } 
	\{ \lambda ^t _{n'} \} \text{ tending to } 0 \text{ as } n' \to \infty  \\
~~ &
 \text{s.t. } \LA d _V \psi (u'_{\lambda ^t _{n'}} (t) ),  u _{\lambda ^t _{n'}}(t) \RA _V
			\to \LA \xi (t) ,  u (t) \RA _V \text{ as } n' \to \infty.
\end{array}
\RD .
\end{equation*}
Since $\LA \xi (\cdot ) , u (\cdot )  \RA _V \in L ^1 (0,T)$,
 almost every point of $(0 , T)$ satisfy the first requirement.
Moreover, thanks to Fatou's lemma and \eqref{St1-05}, 
$t \mapsto \liminf _{n\to \infty} |d _V \psi (u _{\lambda _n } (t))| ^{p'} _{V^{\ast}}$
belongs to $L^1 (0, T)$, and then takes a finite value at a.e. $t \in (0,T )$.
Hence from \eqref{St1-10} and \eqref{St1-11},
 the second requirement is also assured by a.e. $t \in (0,T )$.
 Therefore $\Upsilon $ has full Lebesgue measure in $(0,T)$.

Fix $t _1 , t_2 \in \Upsilon $ with $t _  1 < t _2 $ arbitrarily.
Since $\lambda \partial _V \phi _{\lambda } ( u _{\lambda  } )
			= F _V ( u_{\lambda }  - J _{\lambda } u_{\lambda } ) $
			holds by the definition of the Yosida approximation,
			we have for every $\lambda > 0$
\begin{equation}
\begin{split}
\label{St1-12} 
&\int_{t_ 1}^{ t_2 } \LA \partial _V \phi _{\lambda } (u_{\lambda } (t))  , J_{\lambda } u_{\lambda } (t) \RA _ X dt  
 =\int_{t_ 1}^{ t_2 } \LA \partial _V \phi _{\lambda } (u_{\lambda } (t))  , J_{\lambda } u_{\lambda } (t) \RA _ V dt \\  
 &=\int_{t_ 1}^{ t_2 } \LA \partial _V \phi _{\lambda } (u_{\lambda } (t))  , u_{\lambda } (t) \RA _ V dt   
 - \lambda ^{-1 }\int_{t_ 1}^{ t_2 } |  u_{\lambda }  - J _{\lambda } u_{\lambda } | ^2_V dt \\  
 &\leq \int_{t_ 1}^{ t_2 } \LA \partial _V \phi _{\lambda } (u_{\lambda } (t))  , u_{\lambda } (t) \RA _ V dt .  
\end{split}
\end{equation}
Repeating the same calculation as that for \eqref{St1-01},
i.e., multiplying the equation of (AP)$^h _{\lambda }$ by $u _{\lambda }$
and integrating over $(t_1 , t_2)$,
we get
\begin{equation}
\label{St1-13} 
\begin{split}
 & \varepsilon \int_{t _1}^{t _2 } \LA d_V \psi (u' _{\lambda } (t))  , u '_{\lambda } (t) \RA _V dt
 -\varepsilon \LA d_V \psi (u' _{\lambda } (t _2 ))  , u _{\lambda } (t _2) \RA _V
	 +\varepsilon \LA d_V \psi (u' _{\lambda } (t_1))  , u _{\lambda } (t_1) \RA _V \\
	& +\varepsilon \int_{t _1}^{t _2 } \LA d_V \psi (u _{\lambda } (t)) , u_{\lambda } (t) \RA _V dt 
		+  \int_{t_1}^{t_2}  \LA \partial _V \phi _{\lambda } (u _{\lambda } (t)) , u_{\lambda } (t) \RA _V dt \\
		&\hspace{10mm} +  \varepsilon \int_{t _1}^{t _2 }
				\LA F_V (u _{\lambda } (t)) , u_{\lambda } (t) \RA _V  dt
		=   \int_{t _1}^{t _2 }  \LA  f(t) + h(t) , u_{\lambda } (t) \RA _V dt .
\end{split}
\end{equation}
 Monotonicity of $d _V \psi $ 
    in  $L^p (t _1 , t _2 ; V)$ yields
\begin{equation}
\label{St1-14}
  \liminf_{n\to \infty } \int_{t_1}^{t_2} \LA d _V \psi ( u '_{\lambda _n } (t) ) , u '_{\lambda _n } (t) \RA _V dt \geq 
 \int_{t_1}^{t_2} \LA   \xi (t) ,  u'  (t)  \RA _V dt.
\end{equation} 
 Furthermore, in view of \eqref{conv:u:lambdan}, \eqref{conv:u:lambdan:Lr} 
   and \eqref{identity:F:dVpsi}, 
  we get
\begin{equation}\label{conv:psi:F:ulambdan}
\begin{split}
& \lim_{n\to \infty } \int_{t_1}^{t_2} \LA d _V \psi ( u _{\lambda _n } (t) ) 
     + F_V(u_{\lambda _n}), u _{\lambda _n } (t) \RA _V dt \\
     &\hspace{15mm} 
       = \int_{t_1}^{t_2} \LA   d _V \psi (u(t)) + F_V(u (t) ) , u (t)  \RA _V dt.
\end{split}
\end{equation}
  Thus choosing a suitable subsequence of $\{ \lambda _n\}$ in \eqref{St1-14},
 	denoted by $\{ \lambda _k \}$, and taking the limit in \eqref{St1-13} with 
	$ \lambda = \lambda_k \to 0 \ \text{as} \ k \to \infty $, 
we obtain by \eqref{St1-12}, \eqref{St1-13}, \eqref{St1-14} and \eqref{conv:psi:F:ulambdan} 
\begin{equation}
\label{St1-15} 
\begin{split}
 &\limsup_{k \to \infty }  
\int_{t _1}^{t _2} 
 \LA \partial _V \phi _{\lambda _k } (u _{\lambda _k } (t)) , J _{\lambda 
_k}u_{\lambda _k } (t) \RA _X dt\\
& \leq
- \varepsilon \int_{t _1}^{t _2 } \LA \xi (t)  , u ' (t) \RA _V dt
 +\varepsilon \LA  \xi  (t _2 )  , u  (t _2) \RA _V
	 - \varepsilon \LA \xi (t_1)  , u (t_1) \RA _V \\
	& \hspace{10mm} - \int_{t _1}^{t _2 } \LA \varepsilon  d _V \psi (u(t))
			+ \varepsilon F_V (u  (t)) - f(t) - h(t) , u (t) \RA _V dt .
\end{split}
\end{equation}
From the fact that
$ u = v \in W^{1,p } (0,T ;V) \cap L ^m (0,T ; X )$,
$\xi \in L^{p'} (0,T ;  V^\ast) $, and $\xi ' \in  L^{p' } (0,T ;V^{\ast }) + L ^{m'} (0,T ; X ^{\ast })$,
we can apply Proposition \ref{Pro2.4}
and derive
\begin{equation}
\label{St1-16} 
\begin{split}
 &\limsup_{k \to \infty }  
\int_{t _1}^{t _2} 
 \LA \partial _V \phi _{\lambda _k } (u _{\lambda _k } (t)) , J _{\lambda 
_k}u_{\lambda _k } (t) \RA _X dt\\
& \leq
 \varepsilon  \LA \xi ' (t)  , u (t) \RA _{ L^{p } (t_1,t_2 ;V) \cap L ^{m} (t_1,t_2  ; X ) } \\
& - \int_{t _1}^{t _2 } \LA \varepsilon  d _V \psi (u(t)) + \varepsilon F_V (u  (t)) - f(t) - h(t) , u (t) \RA _V dt .
\end{split}
\end{equation}
Substituting  $\varepsilon \xi'  = \varepsilon  d_V \psi (u) + \varepsilon F_V (u) +  \eta - f -h $ 
  in \eqref{St1-16}, 
we have 
\begin{equation}
\label{St1-17} 
 \limsup_{k \to \infty }  
\int_{t _1}^{t _2} 
 \LA \partial _V \phi _{\lambda _k } (u _{\lambda _k } (t)) , J _{\lambda 
_k}u_{\lambda _k } (t) \RA _X dt
 \leq
  \int_{t _1}^{t _2 } \LA \eta (t) , u (t) \RA _X dt .
\end{equation}
Since 
 $\partial _V \phi _{\lambda _k } (u _{\lambda _k } )  \in 
\partial _V \phi (J _{\lambda _k}u _{\lambda _k } )  \subset \partial _X \phi _X (J _{\lambda _k}u _{\lambda _k } )$
and 
$\partial _X \phi _X  $ is maximal monotone in $L^m (t_1 ,t_2  ; X )$,
then \eqref{St1-17} implies that (recall Proposition \ref{Pro2.2})
\begin{equation}
\begin{split}
&\lim_{ k \to \infty }  
\int_{t _1}^{t _2} 
 \LA \partial _V \phi _{\lambda _k } (u _{\lambda _k } (t)) , J _{\lambda _n}u_{\lambda _k } (t) \RA _X dt
 =
  \int_{t _1}^{t _2 } \LA \eta (t) , u (t) \RA _X dt , \\
&  \eta (t) \in \partial _X \phi _X (u (t))~~~\text{ for a.e. } t\in (t_1 ,t_2 ).
\end{split}
  \label{St1-18} 
\end{equation}
Moreover, since $t _1 , t _  2 $ is chosen arbitrarily in
$\Upsilon $ which  has full Lebesgue measure in $(0,T)$,
we see that $\eta \in \partial _X \phi _X (u) $ in $L^m (0 , T  ; X )$.

Let $t_ 1 , t_2 \in \Upsilon $.
Recall that \eqref{St1-12} and  \eqref{St1-13} give us   
\begin{align*}
 & \varepsilon \int_{t _1}^{t _2 } \LA d_V \psi (u' _{\lambda } (t))  , u '_{\lambda } (t) \RA _V dt \\
 &	\leq 
		\varepsilon \LA d_V \psi (u' _{\lambda } (t _2 ))  , u _{\lambda } (t _2) \RA _V
 					-\varepsilon \LA d_V \psi (u' _{\lambda } (t_1))  , u _{\lambda } (t_1) \RA _V 
-  \int_{t _1}^{t _2}  \LA \partial _V \phi _{\lambda } (u _{\lambda } (t)) , J_{\lambda }u_{\lambda } (t) \RA _X dt 
	 \\
  &\hspace{5mm}   -\int_{t _1}^{t _2 }  \LA  \varepsilon d_V \psi (u _{\lambda } (t)) 
                         +  \varepsilon F_V (u _{\lambda } (t)) - f(t) - h(t) ,
		u_{\lambda } (t) \RA _V dt. 
\end{align*}
Taking the limit  $ \lambda = \lambda_k \to 0$ in the inequality above  and using 
\eqref{St1-18}, the definition of $\Upsilon $, and 
Proposition \ref{Pro2.4} (integration by parts),
we obtain
\begin{equation}
\begin{split}
   &  \varepsilon 
    	\limsup_{k \to \infty } 
	     \int_{t_1}^{t_2 } \LA d_V \psi (u'_{\lambda_k } (t))  , u'_{\lambda_k } (t) \RA_V dt    
\\
   &  \leq 
       \varepsilon \LA \xi (t_2 ), u (t_2) \RA_V
 		  -\varepsilon \LA \xi (t_1 ), u (t_1) \RA_V 
            - \int_{t_1}^{t_2}  \LA  \eta(t), u(t) \RA_X dt  
\\
   &   \hspace{5mm}   -\int_{t_1}^{t_2 }  \LA \varepsilon d_V \psi (u (t) ) 
                                            + F_V (u(t)) - f(t) - h(t), u(t) \RA_V dt  
\\
   &   \leq \varepsilon  \LA \xi', u  \RA_{L^m (t_1,t_2; X) \cap L^p (t_1,t_2; V) } 
             + \varepsilon  \int_{t_1}^{t_2}  \LA \xi, u'  \RA_{ V } dt 
              - \int_{t_1}^{t_2}  \LA  \eta(t), u (t) \RA_X dt                         
\\
   &   \hspace{5mm} -\int_{t_1}^{t_2 }  \LA  \varepsilon d_V \psi (u (t)) 
                           + \varepsilon F_V (u(t)) - f(t) - h(t), u(t) \RA _V dt. 
\end{split}
     \label{est:dVpsi:ulambda:prime} 
\end{equation}
Substituting  $\varepsilon \xi'  - \eta - \varepsilon  d _V \psi (u) - \varepsilon F_V (u) + f + h  =0 $ 
   in \eqref{est:dVpsi:ulambda:prime},
we have 
\begin{equation}\label{ineq:xi:uprime}
 	\limsup_{k \to \infty } 
		\int_{t_1}^{t_2 } \LA d_V \psi (u'_{\lambda_k } (t)), u'_{\lambda_k } (t) \RA _V dt 
		\leq \int_{t_1}^{t_2}  \LA \xi  , u '  \RA _{ V } dt .
\end{equation}
Hence we can repeat the same argument as that for $\eta \in \partial _X \phi _X (u)$
with aid of the maximal monotonicity of $d_V \psi $
and we can show that $\xi = d_V \psi (u)$ in $ L^{p' } (0,T ; V ^{\ast })$.
Therefore, we conclude that $u$ is the desired solution to (AP)$^h $
satisfying \eqref{Regu_S1}.
The uniqueness of solution can be assured by the monotonicity of operators in (AP)$^h$
and Proposition \ref{Pro2.1}. 
\end{proof}


\subsection{Step 2 (Proof of Theorem \ref{Th1}) }

 Let $\Xi := L^{p'} (0,T; V^{\ast}) $ with the weak topology.
   By Step 1, $\beta  $  defined by  the relationship \eqref{Scauder}
     is  well defined as an operator acting on $\Xi $.
In order to apply the Schauder--Tychonoff fixed point theorem,
we shall check that
\begin{itemize}
\item[{\#}1.] Let $K_R := \{ h\in L^{p ' } (0,T ; V ^{\ast});~|h| _{L^{p ' } (0,T ; V ^{\ast})}  \leq R \}$. 
                Then there exists $R >0 $ such that $\beta $ maps $ K_R $ into itself.

\item[{\#}2.] For every sequence $\{ h _n \} _{n \in N}$
which weakly converges to $h $ in $L^{p ' } (0,T ; V ^{\ast})$,
$\{ \beta (h_n )  \} _{n\in \N }$ 
weakly converges to $\beta (h )$ in $L^{p ' } (0,T ; V ^{\ast})$.

\end{itemize}
Then by virtue of Theorem 1 in Arino--Gautier--Penot \cite{Schauder}, 
 $\beta $ possesses at least one fixed point $h_0 \in  L^{p ' } (0,T ; V ^{\ast})$.
 Obviously, $u_ {h _0 }$ is a solution to (AP)$_{\varepsilon }$
satisfying regularities \eqref{Regu_T1} in Theorem \ref{Th1}.

In this subsection, $c_2$ and $C_2$ stand for general constants 
independent $\varepsilon > 0 $ and $h $.

\vspace{2mm} 
(Proof of {\#}1)  Let $\eta_h$  be the section of $ \partial_X \phi_X (u_h)$  
	satisfying (AP)$^h$.
Repeating the same calculation as  that for \eqref{St1-02}, i.e.,
multiplying the equation of  (AP)$^h$ by $u_h $
and using
\begin{align}
   \int_0^T \LA \eta_h(t), u_h(t) \RA_X dt 
       & = \int_0^T \LA \eta_h(t), u_h(t) - v \RA_X dt 
              + \int_{0}^{T}  \LA \eta_h (t), v \RA_X dt   \notag
\\
       & \geq  \int_0^T \phi(u_h(t)) dt - T \phi(v) 
              - |v |_X \int_0^T  | \eta_h (t) |_{X^{\ast}} dt
                \quad  \forall v \in D(\phi)  \label{est:eta:u:X}
\\
       &  \geq c_2 |u_h |^m_{L^m(0,T;X)} - C_2,            \notag
\end{align}	
 we obtain by Young's inequality
\begin{equation}
  \label{St2-02} 
   \begin{split} 
     & c_2 \varepsilon | u _ h | ^p _{W^{1,p} (0,T ;V)}
         + c_2 | u _ h | ^m _{L^{m} (0,T ;X)}
           + \varepsilon  | u _h |^2_{ {L^{2} (0,T ;V )} }
\\
     &~~ \leq 
           \int_{0}^{T} \LA  f(t) + h(t) , u_{\lambda } (t) \RA _V dt 
		      + 2 ( \varepsilon +1 ) ~\! C_2  
\\ 
     &~~ \leq
            C_{2 } \LC | f | ^{p'} _{L^{p'} (0,T ;V^{\ast })} +1 \RC 
                  + \delta  | h | ^{p'} _{L^{p'} (0,T ;V^{\ast })}
                    + C_{\delta }  | u_h | ^{p} _{L^{p} (0,T ;V)},
   \end{split}
\end{equation}
where  $\delta > 0  $ is an arbitrarily fixed constant
and $C_\delta $ is a general constant determined only by $\delta >0$.
Since we assume that $X \hookrightarrow V $ and $m > p $,
we get
\begin{equation*}
C _{\delta }  | u_h | ^{p} _{L^{p} (0,T ;V)} 
  \leq 
C _{\delta }  | u_h | ^{p} _{L^{m} (0,T ;X )} 
  \leq 
\delta  | u_h | ^{m} _{L^{m} (0,T ;X )} 
 + C _{\delta } .
\end{equation*}
Hence for every  $\delta \leq c_2 $,
we have
\begin{equation*}
  c_2 \varepsilon  | u'_ h |^p_{L^{p}(0,T ;V)} \leq  
      C_{\delta }  \LC  1 + | f |^{p'}_{L^{p'}(0,T;V^{\ast})} \RC
        + \delta  | h |^{p'}_{L^{p'} (0,T;V^{\ast })}.
\end{equation*}
 From \eqref{A2} we can derive 
\begin{equation}\label{est:beta:h}
   \begin{split}
   c_2 \varepsilon  | \beta (h) |^{p'}_{L^{p'} (0,T;V^{\ast })}  
      & = c_2 \varepsilon  | d_V \psi (u'_ h) |^{p'}_{L^{p'} (0,T;V^{\ast })} 
\\
      & \leq  C ~\! C_{\delta }  \LC  1 + | f |^{p'}_{L^{p'} (0,T;V^{\ast })} \RC
          +  C ~\! \delta  | h |^{p'}_{L^{p'}(0,T;V^{\ast })} 
            + C ~\! C_\delta \varepsilon ~\! T, 
   \end{split}            
\end{equation}
 where $C$ is the constant appearing in \eqref{A2}.

Here we fix $R$ and $\delta$ such that 
\begin{equation*}
   C ~\! C_{\delta }  \LC  1 + | f |^{p'}_{L^{p'} (0,T;V^{\ast })} \RC 
     + C ~\! C_\delta \varepsilon ~\! T
       = \frac{c_2 ~\!\varepsilon }{2} R^{p'}, 
         \quad \delta = \min ( c_2, \frac{c_2 \varepsilon}{2 ~\! C} ).
\end{equation*}
  Then by \eqref{est:beta:h}, we get 
\begin{equation*}
   | \beta (h) |^{p'}_{L^{p'} (0,T;V^{\ast })} 
     \leq \frac{1}{c_2 \varepsilon} \left( \frac{c_2 \varepsilon}{2} 
                                             + C ~\! \delta \right) R^{p'} 
       \leq R^{p'},
\end{equation*}
which implies $ \beta(h) \in K_R$.
\qed 
 
(Proof of {\#}2) Let $\{ h_n \}_{n \in \N }$ be a sequence such that $h_n \rightharpoonup h$. 
 Let $u_n = u_{h_n}$ be the solution of (AP)$^{h_n}$ and denote by $\eta_n $ the section of
	  $\partial_X \phi_X(u_n) $ satisfying (AP)$^{h_n}$.   
Moreover, let $C' $ denote a general constant independent of $n$, which may depend on $\varepsilon $.
Noting  that $  | h _n  |  _{L^{p'} (0,T ;V^{\ast })} \leq C' $,
 we repeat the same procedure to establish a priori estimates as in the previous step.
Indeed, multiplying the equation of (AP)$^{h_n}$ by $u_n$, we have
(see \eqref{St2-02}) 
\begin{equation*}
 | u_n |^p_{W^{1, p} (0,T ;V)} + | u_n |^2_{L^{2} (0,T ;V)} 
    +  \int_{0}^{T}  \LA \eta_n (t), u _n  (t) \RA_{ X} dt
      \leq C' ,
\end{equation*}
which yields  (see \eqref{est:eta:u:X})
\begin{equation}\label{est:int:phi}
   |F_V (u_n ) |_{L^{p'} (0,T; V^{\ast})} +  
      \int_{0}^{T} \phi ( u_n (t) ) dt \leq C'.
\end{equation}
Then by \eqref{A2}, \eqref{A3}, and \eqref{A7}, we get
\begin{equation*} 
 | d _V \psi ( u  _n  ) |  _{L^{ p'} (0,T ;V^{\ast})} 
+
 | d _V \psi ( u '_n  ) |  _{L^{ p'} (0,T ;V^{\ast})} 
+
 |  u _n   |  _{L^{ m} (0,T ;X)} 
+
 |  \eta  _n   |  _{L^{ m ' } (0,T ;X ^{\ast })} 
 \leq C' .
\end{equation*}
From the boundedness of remainder terms in (AP)$^{h_n}$,
we obtain
\begin{equation*} 
 |  (d_V \psi (u' _n ))'  |  _{L^{ p ' } (0,T ;V ^{\ast }) +L^{ m ' } (0,T ;X ^{\ast })} 
 \leq C' .
\end{equation*}

 Arguments for the convergence given in the previous step can be repeated 
	for this step. In fact, applying Theorem 3 of \cite{Simon}
to $\{ u_n \} _{n \in \N }$, which is uniformly bounded in 
$ W^{ 1 ,p  } (0,T ;V ) \cap L^{ m  } (0,T ;X )$,
we can extract a subsequence denoted by $\{ u_l \} _{l \in \N }$
which strongly converges in $C([0,T] ; V)$.
Let $u$ stands for its limit (remark $u(0) = u(T)$ holds).
Furthermore, there exists a subsequence of  $\{ u_l \} _{l \in \N }$
(we still employ the same index) such that
\begin{align*}
F_V (u _l ) \rightharpoonup \exists w  ~~~&~~\text{ weakly in } L^{p'} (0,T ; V ^{\ast} ),\\
\eta _l
\rightharpoonup \exists \eta
					~~~&~~\text{ weakly in } L^{m'} (0,T ; X ^{\ast} ), \\
d _V \psi  (u _l) \rightharpoonup \exists a
					~~~&~~\text{ weakly in } L^{p'} (0,T ; V ^{\ast} ), \\
d _V \psi  (u ' _l) \rightharpoonup \exists \xi
					~~~&~~\text{ weakly in } L^{p'} (0,T ; V ^{\ast} ), \\
( d _V \psi  (u ' _l)  )' \rightharpoonup  \xi '
					~~~&~~\text{ weakly in } L^{p'} (0,T ; V ^{\ast} )+ L^{m'} (0,T ; X ^{\ast} ).
\end{align*}
  By virtue of the demiclosedness of $F_V$ and $d_V \psi$, we can easily see that $w = F_V (u)$ 
 and $a = d_V \psi (u )$.  Therefore limit functions given above fulfill
the equation
\begin{equation*}
- \varepsilon  \xi '
+ \varepsilon d_V \psi (u) + \eta + \varepsilon F_V (u)
= f + h ~~~\text{in}~~ L^{p'} (0,T ; V ^{\ast} )+ L^{m'} (0,T ; X ^{\ast} ).
\end{equation*}
By the same  reasoning as that for \eqref{St1-11},
we can deduce the strong convergence of $\{ d _V \psi  (u '_l) \} _{l \in \N }$
to $\xi $ in $C([0,T] ; X^{\ast})$ and then $\xi (0) = \xi(T)$.

Define $\Upsilon _h $, which has full Lebesgue measure in $(0,T)$ by 
\begin{equation*}
\Upsilon _h
:=
\LD
\begin{array}{l|l}
~~ &
 t \text{ is a Lebesgue point of } t \mapsto \LA \xi (t) ,  u (t) \RA _V \text{ and } \\
 t \in (0, T ) &
  \{ l \} \text{ has a subsequence } 
	\{ l ^t _{k} \}  \text{ tending to } \infty  \text{ as } k \to \infty  \\
~~ &
 \text{s.t. } \LA d _V \psi (u^{ \prime}_{l ^t _{k}} (t) ),  u _{l ^t _{k}}(t) \RA _V
			\to \LA \xi (t) ,  u (t) \RA _V \text{ as } k \to \infty.
\end{array}
\RD .
\end{equation*}
Let $t_1 < t_2 $ belong to $\Upsilon $.
Tracing the argument from \eqref{St1-13} to \eqref{St1-16},
we can show that
\begin{align*}
&\limsup _{k \to \infty } 
\int_{t_1}^{t_2} \LA  \eta _{ l  _{k}} (t) , u_{ l _k  } (t)  \RA _X dt \\
&~\leq 
\varepsilon \LA  \xi ' (t) , u (t)  \RA _{ L^{p} (t_1 , t_2  ; V  )\cap L^{m} (t_1 , t_2  ; X  )}
 - \int_{t _1}^{t _2 } \LA \varepsilon  d _V \psi (u(t)) + \varepsilon F_V (u  (t)) - f(t) - h(t) , u (t) \RA _V dt \\
&~= 
\int_{t_1}^{t_2} \LA  \eta  (t) , u(t)  \RA _X dt ,
\end{align*}
where $\{ l _k \} _{k \in \N}$ is  a suitable subsequence of $\{ l \}  $.
Hence from Proposition \ref{Pro2.2}, we derive $\eta \in \partial _X \phi _X (u)$ in $L^m (0,T ; X)$.
Moreover,  repeating the same verification as that for 
	\eqref{ineq:xi:uprime}, we can obtain
\begin{equation*}
\limsup _{l \to \infty } \int_{t_1 }^{t_2 }  \LA  d_V \psi (u'_l (t)) , u'_l (t) \RA _V  dt
\leq  
\int_{t_1 }^{t_2 } \LA \xi (t) , u'  (t) \RA _V dt,
\end{equation*}
which implies $\xi = d _V \psi (u' )$ in $L^{p'} (0,T ; V^{\ast })$.

Since the uniqueness of solution to (AP)$^h$ 
assures that the above argument does not depend on the choice of subsequences,
 the original sequence 
$\{ d _V \psi (u' _{n}) \} _{n \in \N} = \{ \beta (h_ n ) \} _{n \in \N} $ 
also converges to  $ d _V  \psi  (u' )  = \beta (h ) $ 
 weakly in $L^{p'} (0,T; V^{\ast})$. 
 \qed 

\begin{remark}   
   In order to derive the uniqueness of the solution of (AP)$_\varepsilon$ by the standard argument, 
     we need to handle the following terms concerning the difference of two solutions 
       $u_1$ and $u_2$: 
\begin{equation*}
 \LA  d_V \psi (u'_1 ) - d_V \psi (u'_2 ), u_1 - u_2  \RA _V,~~~~
   \LA  \partial  \phi(u_1) - \partial \phi  (u_2 ), u'_1 - u'_2  \RA_V. 
\end{equation*} 
  Even for Cauchy problem, assumptions \eqref{A1}-\eqref{A4} are not enough to control these terms. 
   The characterization of Cauchy problem by Euler--Lagrange equation of the WED functional
copes with this difficulty 
arising in the doubly nonlinear evolution equations
and enables us to assure the uniqueness of solution to approximation problem 
when either $\phi $ or $\psi $ is strictly convex
(see Theorem 4.2 of Akagi--Stefanelli  \cite{AS_Cau}).
If one leaves variational approach, however,
the uniqueness can not be assured only by such a natural condition
especially in the case of the time-periodic problem,
where the technical assumptions might be more aggravated than those in Cauchy problem in general.
\end{remark}

In order to  complete the proof of Theorem \ref{Th1}, we verify the following:
\begin{lemma}
\label{Lem3.1} 
There exist a solution $u _{\varepsilon }$ to {\rm (AP)}$_{\varepsilon }$
which satisfies
\begin{equation}
\label{St3-01}
\int_{0}^{T} \LA d _V \psi ( u' _{\varepsilon }(t)  ) , u' _{\varepsilon } (t)\RA _V dt
\leq
\int_{0}^{T} \LA  f(t) , u' _{\varepsilon } (t)\RA _V dt.  
\end{equation}
\end{lemma}
\begin{proof}
We return to Step 1.
Let $u _{\lambda }$ be  the unique solution to (AP)$^{h} _{\lambda }$.
Multiplying (AP)$^{h} _{\lambda }$ by $u ' _{\lambda }$ 
and integrating over $(0,T)$ with respect to $t$, 
we have
\begin{equation*}
- \varepsilon \int_{0}^{T} \LA  (d_V \psi (u' _{\lambda } (t))) ' , u' _{\lambda }(t) \RA _V dt
- \int_{0}^{T}  \LA h(t) , u'  _{\lambda } (t) \RA _V dt
= 
\int_{0}^{T}  \LA f(t) , u'  _{\lambda } (t) \RA _V dt.
\end{equation*}
Here we used Proposition \ref{Pro2.3} (chain rule) and the periodicity of $u _{\lambda }$.

By formal calculation, we get 
\begin{align*}
&-  \int_{0}^{T} \LA  (d_V \psi (u' _{\lambda } (t))) ' , u' _{\lambda }(t) \RA _V dt \\
& = -\LA  d_V \psi (u' _{\lambda } (0))  , u' _{\lambda }(T) -u' _{\lambda }(0)  \RA _V 
+
 \int_{0}^{T} \LA  d_V \psi (u' _{\lambda } (t)) , u'' _{\lambda }(t) \RA _V dt \\
& \geq \psi (u'  _{\lambda }(0)) - \psi (u' _{\lambda } (T))  
+
 \int_{0}^{T} \frac{d}{dt}    \psi (u' _{\lambda } (t))  dt = 0 .
\end{align*}
Although we can not assure the existence of $u''_\lambda(t)$ in a proper way, 
we can rigorously justify the following:
\begin{equation}
- \int_{0}^{T}   \LA  ( d _V \psi ( u ' _{\lambda} (t)  ) )' , u'  _{\lambda } (t) \RA _V dt
\geq 0 . 
\label{St3-02} 
\end{equation}
Indeed, 
by the same procedure as that in the proof for Lemma A.1 in \cite{AS_Cau}, 
one can see that
\begin{equation}
\int_{t_1}^{t_2} \LA (d _V \psi ( u ' _ {\lambda } (t)) )' , u' _{\lambda }(t) \RA _V dt
\leq \psi ^{\ast } (d _V \psi (u' _ {\lambda } (t_ 2 ) )) -\psi ^{\ast } (d _V \psi (u'_ {\lambda }  (t_ 1 ) ))   
\label{LF} 
\end{equation}
holds for every $t _1 , t_2 $ belonging to 
\begin{equation*}
\Theta  _{\lambda }
:=
\LD
\begin{array}{l|l}
~~ &
 ~~~t \text{ is a Lebesgue point of } \\
 t \in (0, T ) &
   ~~~t \mapsto \LA (d _V \psi ( u ' _ {\lambda } (t)) )' , u' _{\lambda }(t) \RA _V  \\
~~&  ~~~\text{and } u_{\lambda } \text{ is differentiable at } t.
\end{array}
\RD ,
\end{equation*}
where $\psi ^{\ast }$ denotes the Legendre--Fenchel transform of $\psi $.
Since $d_V \psi (u' _{\lambda } (t)) \in W^{1, p' } (0,T ; V ^{\ast }) $ and $ u_{\lambda } \in W^{1,p } (0,T ; V )   $
(recall \eqref{Reg3-1}), $\Theta _{\lambda }$ has full Lebesgue measure in $(0,T)$.
Moreover, from the fundamental facts that $\psi ^{\ast } $ is proper l.s.c. convex functional on $V^{\ast }$
and $[v, u] \in \partial _{V^{\ast} } \psi ^{\ast }$ is equivalent to $v = \partial _V \psi (u) =d _V \psi (u) $,
we can derive  the absolute continuity of $\psi ^{\ast } (d_ V \psi (u'_{\lambda } (\cdot )))$ on $[0,T ]$.
Then letting $t_1 \to 0 $ and $t_ 2 \to T $, we obtain \eqref{St3-02}.

Hence $u _{\lambda }$ satisfies  
\begin{equation}
- \int_{0}^{T}  \LA h(t) , u'  _{\lambda } (t) \RA _V dt
\leq 
\int_{0}^{T}  \LA f(t) , u'  _{\lambda } (t) \RA _V dt.
\label{St3-03} 
\end{equation}
Since we already showed that $u ' _{\lambda } \rightharpoonup u ' _ h $ in $L^p (0,T ;V)$ as $\lambda \to + 0 $,
where $u _ h $ stands for the unique solution to (AP)$^h $,
 \eqref{St3-03} leads to
\begin{equation}
- \int_{0}^{T}  \LA h(t) , u'  _h (t) \RA _V dt
\leq 
\int_{0}^{T}  \LA f(t) , u'  _h (t) \RA _V dt
\label{St3-04} 
\end{equation}
 for every $h \in L ^{p '} (0,T ;V ^{\ast})$.
 Now let $h _0  \in L ^{p '} (0,T ;V ^{\ast})$ be a fixed point of $\beta $,
 i.e., $h _ 0 = \beta (h _ 0 ) := - d_ V \psi (u '_{h_0 })$.
Then substituting $h_0$ for $h$ in \eqref{St3-04},
we can  assure that $u _{h _0}$ is one of solutions to (AP)$_{\varepsilon }$
satisfying \eqref{St3-01}.
\end{proof}

\subsection{Step 3  (Proof of Theorem \ref{Th2} : Case $p < m$)}

Henceforth, let $u_{\varepsilon }$ be the solution  of (AP)$_{\varepsilon }$ 
given in Theorem \ref{Th1}
 and $C _3$ stand for a general constant independent of the parameter $\varepsilon \in (0,1)$.
We first note that the definition of the subdifferential and \eqref{A1} yield 
\begin{equation*}
   \LA d_V \psi(u'_\varepsilon(t)), u'_\varepsilon(t) \RA_V 
      \geq  \psi(u'_\varepsilon(t)) - \psi(0) 
         \geq \frac{1}{C} |u'_\varepsilon(t)|_V^p - 1 - \psi(0).
\end{equation*}
  Hence by \eqref{St3-01}, we get 
\begin{equation*}
  \int_0^T |u'_\varepsilon(t) |^p_V dt 
    \leq  C_3 \Bigl( \int_0^T \LA  f(t), u'_\varepsilon(t) \RA_V dt + 1 \Bigr)  
      \leq \frac{1}{2} \int_0^T |u'_\varepsilon(t) |^p_V dt + 
             C_3 \bigl( |f|_{L^{p'}(0,T;V^\ast)}^{p'} + 1 \bigr), 
\end{equation*}
  whence follows the a priori bound for $|u'_\varepsilon|_{L^p(0,T;V)}$. 
   Then by \eqref{A2} and \eqref{A6}, we obtain 
\begin{equation}
  \label{St3-05}
\int_{0}^{T} | u' _{\varepsilon }(t)| ^p _V dt
+
\int_{0}^{T} | d_V \psi ( u' _{\varepsilon }(t) )| ^{p'} _{V^{\ast}} dt
+
\int_{0}^{T}  \psi ( u' _{\varepsilon }(t) ) dt
 \leq C _3.
\end{equation}
Next multiplying (AP)$_{\varepsilon }$ by $u _{\varepsilon }$ and
repeating the same argument as  that for \eqref{St2-02},
we have
\begin{align*}
&  
 \varepsilon | u _{\varepsilon } | ^p _{W^{1,p} (0,T ;V)}
    + | u _{\varepsilon } | ^m _{L^{m} (0,T ;X)}  
      + \varepsilon  | u _{\varepsilon } | ^2 _{L^{2} (0,T ;V )} 
\\
&
  \leq  C_3 \LC | f |^{p'}_{L^{p'} (0,T ;V^{\ast })} 
                 + |d_V \psi (u'_{\varepsilon } (t))|^{p'} _{L^{p'} (0,T ;V^{\ast })} 
                    + | u_{\varepsilon }  | ^{p} _{L^{p} (0,T ;V)}  + 1 \RC
 \\
&
   \leq  C_3 + \frac{1}{2} | u_{\varepsilon }  |^{m} _{L^{m} (0,T ;X)},
\end{align*}
that is,
\begin{equation}
  \label{St3-06} 
 \varepsilon | u _{\varepsilon } | ^p _{W^{1,p} (0,T ;V)}
+
 | u _{\varepsilon } | ^m _{L^{m} (0,T ;X)}  
+
\varepsilon  | u _{\varepsilon } | ^2 _{L^{2} (0,T ;V )} 
\leq 
C _3.
\end{equation}
 Hence the canonical embedding $L^{m} (0,T; X) \hookrightarrow L^{p} (0,T; V) $
and \eqref{St3-05}
yield
\begin{equation}
  \label{St3-07}
   | u _{\varepsilon } | ^p _{W^{1,p} (0,T ;V)} \leq C _3.
\end{equation}
Moreover, from  \eqref{A2} and \eqref{A4}, we can derive
\begin{equation}
  \label{St3-08}
\int_{0}^{T} | \eta _{\varepsilon } (t)| ^{m'} _{X^{\ast }} dt
  +
    \int_{0}^{T}  | d_V \psi (u_{\varepsilon }(t))|_{V^\ast}^{p'}  dt  
      \leq  C_3,
\end{equation}
where $\eta_{\varepsilon }$ is the section of 
   $\partial_X \phi_X(u_\varepsilon(t))$ satisfying (AP)$_{\varepsilon }$.
 Hence by the equation of (AP)$_{\varepsilon }$, we also have
\begin{equation}
  \label{St3-09} 
|\varepsilon ( d_V \psi (u '_{\varepsilon } )) '| _{L^{p'} (0,T; V^{\ast}) + L^{m'} (0,T; X^{\ast}) }
\leq C _3.
\end{equation}

By using \eqref{St3-05}--\eqref{St3-09}, we discuss the convergence of $u_{\varepsilon }$.
To begin with,
\eqref{St3-06} and 
\eqref{St3-07} enable us to apply 
Theorem 3 of  \cite{Simon}
and extract a subsequence $\{ u_ {\varepsilon _ n } \} _{n \in \N}$
which converges  strongly in $C([0,T]; V)$.
Its limit, denoted by $u$, clearly satisfies $u (0) = u (T)$.
We  are going to show that $u$ is the desired periodic solution of (AP).
   By \eqref{St3-07} and \eqref{St3-08}, we get 
\begin{equation*}
| \varepsilon d _V \psi ( u_{\varepsilon })| _{L^{p'} (0 ,T ;V^{\ast })}
+
| \varepsilon F _V  ( u_{\varepsilon })| _{L^{p'} (0 ,T ;V^{\ast })}
\leq 
\varepsilon C _3,
\end{equation*}
which implies that
$\{ \varepsilon _n d _V \psi ( u_{\varepsilon _n}) \} _{n \in \N }$
and 
$\{ \varepsilon _n F _V  ( u_{\varepsilon _n}) \} _{n \in \N }$
 converge to zero  strongly in $L^{p'} (0 ,T ;V^{\ast })$.
Furthermore, there exists a subsequence,
still denoted by  $\{ u_ {\varepsilon_ n } \}_{n \in \N}$,
such that
\begin{align*}
  u_{\varepsilon_n}  \rightharpoonup ~ u 
                     ~~~&~~ \text{ weakly in } L^{m} (0,T ; X ),
\\
 \eta _{\varepsilon _n } \rightharpoonup \exists \eta
					  ~~~&~~ \text{ weakly in } L^{m'} (0,T ; X ^{\ast} ), \\
 d _V \psi  (u ' _ {\varepsilon  _n }) \rightharpoonup \exists  \xi
					~~~&~~ \text{ weakly in } L^{p'} (0,T ; V ^{\ast} ), \\
 \varepsilon ( d _V \psi  (u ' _{\varepsilon _n })  )' \rightharpoonup \exists \zeta 
					  ~~~&~~ \text{ weakly in } L^{p'} (0,T ; V ^{\ast} )+ L^{m'} (0,T ; X ^{\ast} ).
\end{align*}
For every $v \in W^{1,m} (0,T ;X )$  with $v ( 0) = v (T) $, 
  we get by \eqref{St3-05} 
\begin{equation*}
  \LA  \varepsilon_n (d_V \psi (u'_{\varepsilon_n } ))', v  
   \RA_{L^{p}(0,T;V) \cap L^{m}(0,T; X)} 
     = - \varepsilon_n  \int_{0}^{T} \LA  d_V \psi(u'_{\varepsilon_n }(t)), v'(t) \RA_V dt
      \to 0
\end{equation*}
as $n \to \infty $ (use Proposition \ref{Pro2.4}).
Then $\zeta  = 0$ holds by the density.
Therefore,  we obtain $\xi + \eta = f $.
Immediately, $\eta $ belongs to $ L^{p'} (0,T ; V ^{\ast })$, 
since the remainders $\xi $ and $f$ are both members of  $ L^{p'} (0,T ; V ^{\ast })$.

Multiplying (AP)$_{\varepsilon _n}$ by $u_{\varepsilon _n}$,
  we have
\begin{align*}
&\int_{0}^{T} \LA \eta _{\varepsilon _n} (t) , u_{\varepsilon _n} (t) \RA _X dt \\
&=-\varepsilon _n  \int_{0}^{T} \LA d_V \psi (u' _{\varepsilon _n}  (t)) , u' _{\varepsilon _n} (t) \RA _V dt
 -\varepsilon _n  \int_{0}^{T} \LA d_V \psi (u _{\varepsilon _n}  (t)) , u _{\varepsilon _n} (t) \RA _V dt \\
&~~~~ -\varepsilon _n  | u _{\varepsilon _n} | ^2 _{L^2 (0,T ;V )} 
 - \int_{0}^{T} \LA d_V \psi (u ' _{\varepsilon _n}  (t)) , u _{\varepsilon _n}(t)  \RA _V dt
  + \int_{0}^{T} \LA f(t) , u _{\varepsilon _n} (t) \RA _V dt.
\end{align*}
Taking the limit as $n \to \infty$,
we obtain
\begin{align*}
\limsup _{n \to \infty }
\int_{0}^{T} \LA \eta _{\varepsilon _n} (t) , u_{\varepsilon _n} (t) \RA _X dt 
&=
 - \int_{0}^{T} \LA \xi (t) , u (t)  \RA _V dt
  + \int_{0}^{T} \LA f(t) , u (t) \RA _V dt \\
&= 
  \int_{0}^{T} \LA \eta (t) , u (t) \RA _V dt \ 
     = \int_{0}^{T} \LA \eta (t) , u (t) \RA _X dt,
\end{align*}
which implies $\eta \in \partial _X \phi _X( u )$ thanks to Proposition \ref{Pro2.2}.
  Hence  by virtue of \eqref{sub} together with the fact that $\eta \in L^{p'}(0,T;V^\ast)$, 
    we can conclude 
\begin{equation}\label{eta:in:partialphi}
   \eta
 \in \partial_V \phi ( u ) .
\end{equation}
Finally, letting $\varepsilon _ n \to 0$ in \eqref{St3-01} of  Lemma \ref{Lem3.1},
we can see that
\begin{align*}
\limsup _{n \to \infty }
\int_{0}^{T} \LA d_V \psi (u'_{\varepsilon _n} (t)) , u' _{\varepsilon _n} (t) \RA _V dt 
&  \leq 
\int_{0}^{T} \LA f(t) , u'  (t) \RA _V dt \\
&=
\int_{0}^{T} \LA f (t) - \eta  (t) , u'  (t) \RA _V dt 
=
\int_{0}^{T} \LA \xi (t) , u'  (t) \RA _V dt. 
\end{align*}
Here we used Proposition \ref{Pro2.3}  and \eqref{eta:in:partialphi}.
  Then maximal monotonicity of $d_V \psi $ leads to $\xi = d _V \psi (u' )$ in $L^{p'} (0,T ; V ^{\ast})$,
    hence it follows that $u $ is a solution to (AP).
 
 Furthermore Proposition \ref{Pro2.3} together with \eqref{St3-07} and 
 	\eqref{eta:in:partialphi} assures that $\phi(u(t))$ is absolutely continuous on $[0,T]$ 
 	   and hence \eqref{A3} assures that $ u \in L^\infty(0,T;X)$. 
\qed

\subsection{Step 4 (Proof of Theorem \ref{Th2} : Case $ m \leq p$)}
 In this subsection, we consider the excluded case,
 i.e., the case where $m \leq p$. 
Put
\begin{equation*}
\Phi  := \phi + \frac{\mu }{1+ \alpha } \phi ^{1+ \alpha }, 
\end{equation*}
where $\mu \in (0,1 ) $ and $\alpha $ is some fixed exponent with $\DS \alpha > p/m -1 $.
It is easy to see that $\Phi $ is proper l.s.c. convex functional over $V$.
 Since $\phi $ satisfies \eqref{A3},  
\begin{align*}
\Phi  (u) & \geq c  _{\mu} |u| ^m _X(1 + |u|^{m\alpha } _X) -C  _{\mu} \\
& \geq c _{\mu} |u|^{m ( \alpha + 1 ) } _X -C  _{\mu}
\end{align*}
holds for every $u\in D(\Phi ) = D(\phi )$ with some constants $c  _{\mu} ,C  _{\mu}> 0$
(which may depend on the parameter $\mu$).
Moreover, since 
\begin{equation}
\partial _V \Phi 
=
\partial _V  \phi + \mu \phi ^{\alpha } \partial _V \phi ,
~~
\partial  _X \Phi  _X 
=
\partial  _X  \phi _X + \mu \phi ^{\alpha }   _X \partial  _X \phi _X 
\label{Lemma2} 
\end{equation}
(whose proof will be given in Appendix),
then for any $\LB u , \eta \RB \in \partial _X  \Phi _X $ there exist some $C > 0$ 
independent of $\mu $ such that
\begin{align}
|\eta | ^{m' \LC \frac{\alpha +1}{ m' \alpha +1} \RC} _{X ^{\ast}}
\leq 
C \LB |u | ^m _X \LC 1 + |u |^{\alpha m m' } _X  \RC  \RB ^{\frac{\alpha +1}{ m' \alpha +1} }
\leq C ( |u| ^{m (\alpha +1 )}_X +1 ), 
\end{align}
since $\phi $ satisfies \eqref{A4} and \eqref{A8}.
Hence  $\Phi $ fulfills  \eqref{A3} and \eqref{A4} with $\bar{m} := m (\alpha +1 ) > p $
(note that the H\"{o}lder conjugate of $\bar{m} $ coincides with $\DS \bar{m}' = m' (\alpha +1) /(  m' \alpha +1 )$).
Then we can carry out the same argument given above 
with $\phi$ replaced by $\Phi$.
That is to say,
for every
 $ \mu \in (0,1 ) $ and $f \in L^{p ' } (0, T ; V ^{\ast })$,
time-periodic problem
\begin{equation*}
  \text{\rm (AP)}^{\star}_{\mu}
   \begin{cases}
       ~~ d_V \psi(u'_{\mu}) + \LC 1 + \mu \phi^{\alpha} (u_{\mu}) \RC 
            \partial_V \phi(u_{\mu}) \ni f,
	      	  ~~ & ~~~~t \in (0,T)~~~\text{in} \ V^{\ast},
\\
       ~~ u_{\mu} (0) = u_{\mu}(T), 
              ~~ & ~~
   \end{cases}
\end{equation*}
possesses  at least one solution satisfying
\begin{equation}
\begin{split}
\label{Regu_Lem3}
&u  _{\mu} \in W ^{1,p} (0,T ;V) \cap L^{\bar{m} } (0,T ;X ), \\
& d_V \psi (u   '   _{\mu}  ) , ~(1 + \mu \phi ^{\alpha } (u  _{\mu} )) \eta _{\mu} \in  L^{p'} (0,T ;V ^{\ast } ), 
\end{split}
\end{equation}
where $\eta _ {\mu } \in \partial _V \varphi  (u _{\mu  })$.

  We shall establish a priori estimates of $u_{\mu}$ independent of $\mu$  
    by repeating the same calculation in Step 3.
  Let  $c_4 , C_4 >0 $  denote general constants independent of $ \mu \in (0,1)$.
    Multiplying (AP)$^{\star  }_{\mu }$ by $u ' _{\mu }$, using the chain rule
$\DS  \Phi ' (u(t)) = \LA (1 + \mu \varphi ^{\alpha } ( u _{\mu} (t))  ) \eta _{\mu } , u' _{\mu}(t) \RA _V $,
and integrating over $[0,T ]$, we have  by \eqref{A1}, \eqref{A2} and \eqref{A6}  (see \eqref{est:below:psi:v})
\begin{equation}
\int_{0}^{T} | u' _{\mu }(t)| ^p _V dt
+
\int_{0}^{T} | d_V \psi ( u' _{\mu }(t) )| ^{p'} _{V^{\ast}} dt
+
\int_{0}^{T}  \psi ( u' _{\mu }(t) ) dt
 \leq C _4.
\label{St4-01} 
\end{equation}
By the equation of (AP)$^{\star } _{\mu }$,
we obtain 
\begin{equation}
| \eta _{\mu }| _{L ^{p '} (0,T  ; V^{\ast })} 
  + |\mu \phi ^{\alpha } (u_{\mu} ) \eta _{\mu }| _{L ^{p '} (0,T  ; V^{\ast })} 
   \leq 
      |d_V \psi (u ' _{\mu} )| _{L ^{p '} (0,T  ; V^{\ast })}  
       + |f| _{L ^{p '} (0,T  ; V^{\ast })} 
         \leq C_4 .
\label{St4-02} 
\end{equation}
From \eqref{A8}, \eqref{AW2} and the canonical embedding $|\eta _\mu | _{V^{\ast}} \geq c _4 |\eta _\mu | _{X^{\ast}}  $, 
we can derive 
\begin{equation}
\int_{0}^{T} \phi ^{p' / m'} (u _{\mu } (t) ) dt
\leq C_4 .
\label{St4-03} 
\end{equation}
By Proposition \ref{Pro2.3}, $\phi (u _{\mu} (\cdot ))$ is absolutely continuous on $[0,T]$
and hence
there is $t_0 = t^{\mu}_0 \in [0,T ] $
at which $\phi (u _{\mu} (\cdot ))$ attains its minimum.
Immediately,  \eqref{St4-03} implies $\phi  ( u _{\mu} (t _0 ) ) \leq C_4 $.
Hence testing (AP)$^{\star  }_{\mu }$ by $u ' _{\mu }$ again and 
integrating over $[t_0 ,t  ]$ with $t \in [t_0 , t_0 +T ]$, we obtain
\begin{equation}
\sup _{0 \leq t \leq T} \phi (u _{\mu} (t)) 
\leq
\sup _{0 \leq t \leq T} \Phi (u _{\mu} (t)) \leq C_4.
\label{St4-05} 
\end{equation}
Combining \eqref{St4-05} with \eqref{A3} and \eqref{A4},
we can show that
\begin{equation}
c_4 \sup _{0 \leq t \leq T} |u _{\mu } (t)  | _{V}
\leq 
\sup _{0 \leq t \leq T} |u _{\mu } (t)  | _{X}
+
\sup _{0 \leq t \leq T} |\eta  _{\mu } (t)  | _{X ^{\ast }}
\leq 
C_4 .
\label{St4-06} 
\end{equation}

 Note that \eqref{St4-01}, \eqref{St4-02} and \eqref{St4-06}
enable us to follows the same convergence argument as that developed in Step 3.
Indeed, \eqref{St4-02} and 
\eqref{St4-05} yield
\begin{equation}
|\mu \phi ^{\alpha  } (u_{\mu} ) \eta _{\mu} | _{L^{p '} (0,T ; V^{\ast })}
\leq 
\mu C_4 \to 0 ~~~\text{ as } \mu \to 0.
\label{St4-07} 
\end{equation}
Applying Theorem 3 of Simon \cite{Simon} and standard argument,
we can extract a subsequence of $\{ u _{\mu} \} _{\mu \in \N }$ (we skip relabeling)
such that 
\begin{align*}
u_ {\mu} ~\to ~  u~~~&~~\text{ strongly in } C([0,T] ; V ) ,\\
~~~~~&~~\text{ weakly in } W^{1,p}(0,T ; V ) ,\\
~~~~~&~~\ast\text{-weakly in } L^{\infty }(0,T ; X ) ,\\
d _V \psi (u '_ {\mu}) ~\rightharpoonup  ~\xi ~~~&~~\text{ weakly in } L^{p'}(0,T ; V ^{\ast} ) ,\\
\eta _ {\mu}  ~\rightharpoonup  ~\eta ~~~&~~\text{ weakly in } L^{p'}(0,T ; V ^{\ast} ), 
\end{align*}
and $\xi + \eta = f $ in $L^{p'} (0,T ; V^{\ast})$.
Using \eqref{St4-07}, we have 
\begin{align*}
\limsup _{\mu \to 0} \int_{0}^{T} \LA \eta _{\mu} (t) , u_{\mu} (t) \RA _V dt 
&=- \int_{0}^{T} \LA \xi (t) , u  (t) \RA _V dt   + \int_{0}^{T} \LA f (t) , u  (t) \RA _V dt \\
&= \int_{0}^{T} \LA \eta  (t) , u  (t) \RA _V dt  ,
\end{align*}
which  together with Proposition \ref{Pro2.2} implies $\eta \in \partial _V \phi (u) $. 
Furthermore, since  
\begin{equation*}
  \begin{split}
    \int_{0}^{T} \LA d _V \psi (u '_ {\mu} (t))  , u '_{\mu} (t) \RA _V dt 
     & = \Phi(u_\mu(0)) - \Phi(u_\mu(T)) +   \int_{0}^{T} \LA f(t) , u'_\mu(t) \RA _V dt 
 \\
     &  \to \phi(u(0)) - \phi(u(T)) +   \int_{0}^{T} \LA f(t) , u'(t) \RA _V dt 
     	 \quad \text{as} \ \mu \to 0,
  \end{split}
\end{equation*} 
we obtain 
\begin{align*}
\limsup _{\mu \to 0}
\int_{0}^{T} \LA d _V \psi (u '_ {\mu} (t))  , u '_{\mu} (t) \RA _V dt 
&=
 - \int_{0}^{T} \LA \eta (t) , u' (t)  \RA _V dt
  + \int_{0}^{T} \LA f(t) , u '(t) \RA _V dt \\
&= 
  \int_{0}^{T} \LA \xi (t) , u '(t) \RA _V dt,
\end{align*}
which  together with Proposition \ref{Pro2.2} leads to $\xi = d _V \psi (u ' ) $. 
Thus it is shown that $u$ gives a solution of (AP) satisfying \eqref{Regu_T2} .
\qed 


\section{Structural stability} 
In this section, we  show that the method 
developed in the previous section is
 applicable to  the study for the structural stability for solution to (AP). 
More precisely, we here consider 
some perturbation to the functionals $\phi $ and $\psi $ 
in the following sense
(see Definition 3.17 and Proposition 3.19 in Attouch \cite{Attouch}).
\begin{definition}
\label{Mosco} 
Let $Z$ be a reflexive Banach space and 
 let $\varphi$ and $ \varphi_n  \ (n \in \N) $
be proper lower semi-continuous convex functions 
from $Z$ to $( - \infty ,  + \infty ] $.
Then  it is said that $\varphi_n$ converges to $\varphi$ 
	in the sense of Mosco, if the following two conditions (a) and (b) hold.   
\begin{itemize}
\item[(a)] (Liminf condition) If $u _ n \rightharpoonup u $ weakly in $Z$,
then $\liminf _{n \to \infty} \varphi _n (u _ n ) \geq \varphi ( u )$ holds.

\item[(b)] (Existence of recovery sequence)
For every $u \in D(\varphi )$, there exists a sequence $\{ u _n \} _{ n \in \N }$
such that $u _ n \to u $ strongly in $Z$ and $\varphi _ n (u_ n ) \to \varphi (u) $.
\end{itemize}
\end{definition}
We here present the following fact
 (see 
Theorem 3.66 and Proposition 3.59  in \cite{Attouch}),
 which gives a generalization of Proposition \ref{Pro2.2}.
\begin{proposition}
\label{Prop4} 
Let  
$\{ \varphi _n \} _{n \in \N}$
be a sequence of proper l.s.c.\ convex functions on a reflexive Banach space $Z$
which converges to a proper l.s.c.\ convex  function 
 $\varphi : Z \to ( -\infty , + \infty ]$ on $Z$  in the sense of Mosco.
Assume that $ [ u _ n , v _ n ] \in \partial  _Z  \varphi _ n $
satisfies $ u_ n \rightharpoonup   u$ weakly in $Z$,  $ v_ n \rightharpoonup   v$ weakly in $Z^{\ast }$, 
and 
\begin{equation*}
\limsup _{n\to \infty } \LA  v _ n , u _n   \RA _Z \leq \LA v  , u  \RA _Z .
\end{equation*}
Then $[u ,v ] \in \partial _Z \varphi $ and 
$\LA  v _ n , u _n   \RA _Z \to \LA v  , u  \RA _Z $ as $n \to \infty $.
\end{proposition}

To state our result, we introduce the following growth conditions 
	on $\phi_n $ and $\psi_n $, which is similar to (A.1):

\begin{itemize}
\item[(A.2)] 
$\{ \phi _n \} _{n \in \N }$ and $\{ \psi _n \} _{n \in \N }$
are sequences of proper lower semi-continuous convex functionals 
and G\^{a}teaux differentiable convex  functionals over $V$,  respectively  
such that
$ \phi_n $ and $ \psi_n $ converge to $\phi$ and $\psi$ in the sense of Mosco on $V$, 
respectively.
Furthermore there exist some constants $C >0 $ independent of the index $n$ satisfying
\begin{align}
& | u | ^p _V \leq C ( \psi _ n (u) + 1) ~~~\forall u \in V,
	\label{A11}  \\
& | d _V \psi _ n (u)  | ^{p'} _{V ^{\ast }} \leq C( |u |^p _V + 1 ) ~~~\forall u \in V,
	\label{A12}  \\
& | u | ^m _X \leq  C ( \phi _ n (u) + 1)  ~~~\forall u  \in D(\phi _ n ),	
	\label{A13}  \\
& | \eta   | ^{m'} _{X ^{\ast }} \leq C ( |u |^m _X + 1) ~~~\forall \LB u  ,\eta \RB \in \partial _X (\phi _ n)  _X,
	\label{A14} 
\end{align}
 where $p, m \in (1, \infty)$
 and  $(\phi_n)_X$ stands for the restriction of $\phi_n$ onto $X$.
\end{itemize}

   Then our result is stated as follows.
\begin{theorem}
\label{Th4} 
 Assume {\rm (A.0), (A.1)} and {\rm (A.2)} and let 
    $u_n $ be solutions of 
 \begin{equation*}
\text{\rm (AP)}_n 
   \begin{cases}
      ~d \psi_n (u'_n(t)) + \partial \phi_n (u_n(t) ) \ni f_n(t),
			~~& t \in (0,T ) ~~\text{ in } V ^{*}, 
\\
      ~u _ n (0) = u _ n (T), 
            ~~& 
\end{cases}
\end{equation*}
where $\{ f_n \}_{n \in \N }$ and $f$ satisfy either of 
\begin{itemize}
\item[i) ]   $ f_n \to f $ strongly in $L^{p'}(0,T; V^{\ast})$,

\item[ii) ]  $f_n \rightharpoonup f $ weakly in  $W^{1,p'}(0,T; V^{\ast})$
               and $ f(0) = f(T)$, $f_n (0) = f_n (T)$ hold for any $n$.
\end{itemize}
   Then there exist a subsequence $\{ u_{j} \}_{j \in \N } := \{ u_{ n_j} \}_{j \in \N } $
     and its limit $u $ such that
\begin{align*}
   u_ j ~\to ~  u~~~&~~\text{ strongly in } C([0,T]; V ), 
\\
               ~~~~~&~~\text{ weakly in } W^{1,p}(0,T; V ), 
\\
               ~~~~~&~~\ast\text{-weakly in } L^{\infty }(0,T; X ), 
\\
   d_V \psi_j (u_j) ~\rightharpoonup  ~d_V \psi (u)
                 ~~~&~~\text{ weakly in } L^{p'}(0,T; V^{\ast} ), 
\\
   \eta_j  ~\rightharpoonup  ~\eta 
                 ~~~&~~\text{ weakly in } L^{p'}(0,T; V^{\ast} ), 
\end{align*}
and $u $ is a solution to  {\rm (AP)}. 
  Here $\eta $ and  $\eta_j= \eta_{n_j}$ denote 
    the sections of $\partial \phi (u)$ and $\partial \phi_{n_j} (u_{n_j})$ satisfying the equation of 
      {\rm (AP)} and {\rm (AP)}$_{n_j}$, respectively.
 
\end{theorem}

\begin{proof}

   It is clear that assumptions \eqref{A11} and \eqref{A12} yield \eqref{A5}--\eqref{AW1}
     with $\psi $ replaced by $\psi_n $ and  $C$ independent of $n$,  
         and \eqref{A7} with $\phi$ replaced by $\phi_n$ is a direct consequence of 
        	\eqref{A13} and \eqref{A14}.
    Moreover, \eqref{A8}--\eqref{AW2} also hold 
     with $\phi $ replaced by $\phi_n$ by virtue of (b) of Definition \ref{Mosco}.
Indeed, let $v \in D(\phi )$, then there exists 
  a recovery sequence $\{ v _ n \} _{n \in \N }$
    which satisfies $ v_n \to v$  in $V$ and $\phi_n(v_n) \to \phi(v)$.
      Remark that \eqref{A13} implies uniform boundedness of $| v_n |_{X} $.
        By the definition of subdifferential,
\begin{align*}
  \phi_n (u) & \leq \phi_n (v_n ) + \LA \eta, u - v_n  \RA_X
   \leq \phi_n (v_n ) +  | \eta  |_{X^{\ast} }  ( |u|_X  +|v_n |_X  ) 
\\ 
             & \leq C +  |\eta|_{X^{\ast} }  ( |u|_X  + C  ),
\end{align*}
 where $[ u, \eta ] \in \partial_X (\phi_n)_X$ and $C$ is independent of $n$.
   Hence $\phi _n$ fulfills \eqref{A8} by \eqref{A14} and \eqref{AW2} by  \eqref{A13}.

   Then repeating exactly the same manipulations as those for \eqref{St4-01}--\eqref{St4-06}
     in Step 4 of the previous section, we can establish the following estimates for solutions
        $u_n$ of (AP)$_n$ :
\begin{align*}
|u ' _n | _{L^p (0, T ; V )} +
	|d _V \psi _n  ( u ' _n ) | _{L^p (0, T ; V )} + 
		 \int_{0}^{T} \psi _n ( u'_n (t)) dt \leq C _5 ,\\
\sup _{0\leq t \leq T } |u  _n  (t)| _{X} +
	\sup _{0\leq t \leq T }  |\eta  _n  (t) | _{X ^{\ast}} + 
		 \sup_{0\leq t \leq T }  \phi_n ( u_n (t))  \leq C _5 , 
\end{align*}
where $C _ 5 >0 $ is a general constant independent of $n$.
Therefore we can reprise the same discussion of convergence 
 as that in Step 4 of the previous section and obtain
\begin{equation}
\begin{split}
\limsup _{n\to \infty } \int_{0}^{T} \LA  \eta  _n ,  u  _n (t ) \RA _V dt \leq 
   \int_{0}^{T} \LA \eta (t) ,  u (t ) \RA _V dt , \\
\limsup _{n\to \infty } \int_{0}^{T} \LA d _V \psi _n ( u ' _n (t ) ) ,  u ' _n (t ) \RA _V \leq 
   \int_{0}^{T}  \LA  \xi (t) ,  u '  (t ) \RA _V dt, 
\end{split}
\label{S4-01} 
\end{equation}
where $\eta $ and $\xi$ are limits of (a suitable subsequence of) 
$\{ \eta  _n \} _{n \in \N}$
and
$\{  d_V \psi_n ( u ' _n (t ) ) \} _{n \in \N}$, respectively.
  Here assumption i) on $f_n$ is used  to assure   
\begin{equation*}
\int_{0}^{T} \LA f _n (t) , u ' _ n (t) \RA _V  dt 
\to 
\int_{0}^{T} \LA f  (t) , u '  (t) \RA _V dt 
\end{equation*}
as $n \to \infty $ and obtain \eqref{S4-01}.
Assumption ii) also leads to 
\begin{align*}
\int_{0}^{T} \LA f _n (t) , u ' _ n (t) \RA _V  dt 
& =  \LA f _n (T) , u  _ n (T) \RA _V -  \LA f _n (0) , u  _ n (0) \RA _V
 - \int_{0}^{T} \LA f ' _n (t) , u  _ n (t) \RA _V  dt \\
& \to - \int_{0}^{T} \LA f '  (t) , u   (t) \RA _V  dt 
 =   \int_{0}^{T} \LA f  (t) , u ' (t) \RA _V  dt .
\end{align*}
Therefore, it follows our result with the aid of Proposition \ref{Prop4}.
\end{proof}


\section{Application} 

   Let $\Omega \subset \R^d $ be a bounded domain with sufficiently smooth boundary $\partial \Omega $.
     As in Akagi--Stefanelli \cite{AS_Cau}, we consider the following  doubly nonlinear parabolic equation: 
\begin{equation*}
   \text{\rm (DNP)} 
    \begin{cases}
     ~\alpha ( u'(x,t)) - \Delta^a_m u(x,t) = f (x,t)~~ 
        & (x,t) \in \Omega \times (0,T),
\\
     ~u  (x ,t) = 0, ~~
        & (x,t) \in \partial \Omega \times (0,T),
    \end{cases}
\end{equation*}
where 
\begin{equation*}
   \Delta^a_m u (x) := \nabla \cdot \LC a (x) |\nabla u (x)|^{m-2} \nabla u(x) \RC ,~~~
     x \in \Omega , ~ m \in ( 1,  \infty ) ,~~a : \Omega \to \R .
\end{equation*}
   In \cite{AS_Cau}, it is assumed that  $a $,  $\alpha $,  and the exponent $m$ satisfy 
\begin{itemize}
\item[(a.1)] There exist some constants $a _ 1 , a _ 2 > 0$ such that 
               $ a_1 \leq a (x) \leq a_2 $ for a.e. $x \in \Omega $. 

\item[(a.2)] $\alpha : \R \to \R $ is a single-valued maximal monotone function 
               with $D(\alpha ) = \R$.
                 Moreover, there exist $p \in (1 , \infty )$ and constants $c , C > 0 $ such that
\begin{equation*}
   c ~\! |s |^p - \frac{1}{c} \leq A (s) ,~~ |\alpha (s) |^{p'} \leq C (|s|^p  +1 )
      ~~~\forall s \in \R, 
\end{equation*}
where $\DS A (s) := \int_{0}^{s} \alpha (\sigma ) d \sigma $ (primitive function of $\alpha $). 

\item[(a.3)] $\DS p < m^{\ast } := \frac{Nm}{(N - m )_+ } $.

\end{itemize}

To reduce (DNP) to (AP), set 
\begin{equation*}
V := L^ p (\Omega ) , ~~X := W ^{1, m } _0 (\Omega )
\end{equation*}
and  define functionals $\psi$ and $\phi$  on $V$ by
\begin{equation}\label{def:psi:phi}
\begin{split}
\psi (u) & := \int_{\Omega } A(u(x) ) dx, 
\\[2mm]
\phi (u) & : = \begin{cases}
~~\DS  \frac{1}{m}  \int_{\Omega } a(x) |\nabla u(x)|^m dx ~~~
& \text{ if } u \in W^{1,m}_0(\Omega), 
\\
~~ + \infty  ~~~
&\text{ otherwise. }
\end{cases}
\end{split}
\end{equation}
Note that $V $, $X$, and $V^{\ast } = L ^{p'} (\Omega )$ are uniformly convex
and the embedding $X \hookrightarrow V $ is compact by (a.3).
From the growth  condition assumed in (a.2), we have $D( \psi ) = L^{p} (\Omega )$.
It is easy to obtain the convexity and differentiability of  $\psi $ on $V$ 
and see that its derivative coincides with $d_V \psi (u) = \alpha (u)$.
Since $a$ is  assumed to be a non-degenerate coefficient on $\Omega $,
we can show that $\phi _X $ is differentiable on $X $ and $\partial _X \phi _X (u) =  - \Delta ^a _ m u $
by the standard variational argument
(immediately $\partial _V \phi (u) = -\Delta ^a _ m u$ by $\partial _V \phi \subset \partial _X \phi _X$). 
Moreover, assumptions (a.1) and (a.2) 
lead to growth conditions \eqref{A1}--\eqref{A4}.

Therefore, (A.0) and (A.1) are verified for (DNP), 
so Theorem \ref{Th2} assures that the following result holds.
\begin{theorem}
	\label{Ex01} 
	Assume {\rm (a.1)--(a.3)}.
	Then for every $f \in L^{p'}(0,T;L^{p'} (\Omega ))$,
	{\rm (DNP)} possesses at least one time-periodic solution satisfying
	\begin{align*}
	&  u \in W^{1,p}(0,T;L^p(\Omega )) \cap  C([0,T]; W^{1,m}_0(\Omega)) , 
	\\
	&  \alpha (u'), \Delta^a_m u \in L^{p'} (0,T;L^{p'} (\Omega )) .
	\end{align*}
\end{theorem}
	\begin{proof}
		The assertions above follow from the direct application of 
		Theorem \ref{Th2} except $u \in C([0,T]; W^{1,m}_0(\Omega))$. 
		 To verify this, we first note that (a.1) assures that 
		  $L^m_a(\Omega)$ with norm $ |w|_{L^m_a} := (\int_\Omega a ~\! |w|^m ~\!  dx)^{1/m}$ 
		    is uniformly convex and $m ~\! \phi(u(t)) = |\nabla u(t)|_{L^{m}_a}^m$ 
		     is (absolutely) continuous on $[0,T]$. Therefore we easily derive  
		       $ \nabla u \in C([0,T];(L^m_a(\Omega) )^d)$, which together with (a.1) assures 
		         $u \in C([0,T];W^{1,m}_0(\Omega))$.
	\end{proof}

We next consider the structural stability  of (DNP). 
For this purpose, we introduce the following conditions:
\begin{itemize}
	\item[(a.4)]  $\{ a_n \}_{n \in \N }$ is a sequence of functions $a _n : \Omega \to \R $
	such that $a_n (x) \to a (x) $ as $n \to \infty $ for a.e. $x \in \Omega $.
    In addition, (a.1)  with $a$ replaced by $a_n$ holds for all $n$.

	\item[(a.5)] $\{ \alpha_n \}_{n \in \N }$ is a sequence of functions $\alpha_n : \R \to \R$
	such that the sequence of primitive functions $\DS A _n (s) := \int_{0}^{s} \alpha _n (\sigma ) d \sigma  $
	converges to $\DS A (s) := \int_{0}^{s} \alpha  (\sigma ) d \sigma $ in the following sense:
	\begin{itemize}
	\item [i)] If $s _n \to s $ in $\R $, then $\liminf _{n\to \infty } A _n( s_ n ) \geq A (s) $.

	\item [ii)] For every $s \in \R$, there exists a sequence $\{ s _ n \} _{n\in \R }$
	such that $A _n(s_n ) \to A(s)$ as $n\to \infty $.	
	\end{itemize}
	In addition, (a.2) holds with functions $\alpha $ and $A$ replaced by $\alpha _n $ and $A _n$, respectively
	by the same $c, C >0 $ (independent of $n$).

\end{itemize}

 Define $\psi_n$ and $\phi_n$ by \eqref{def:psi:phi} with $A(s)$ and $a(x)$ replaced by 
	           $A_n(s)$ and $a_n(x)$, respectively. 
Then according to \cite{AY} and standard facts in \cite{Attouch},
(a.4) and (a.5) assure Mosco convergence of $\{ \psi _n \} _{n \in \N}$ and $\{ \phi _n \} _{n \in \N}$
to $\psi $ and $\phi $ on $V$, respectively.
Hence it follows from Theorem \ref{Th4} that
\begin{theorem}
\label{Ex02} 
   Assume {\rm (a.1)--(a.5)} and 
     $\{ f_n \}_{n \in \N }$ and $f$ satisfy either of 
\begin{itemize}
\item[i)]    $ f_n \to f $ strongly in $L^{p'} (0,T;L^{p'}(\Omega ))$,

\item[ii)]   $f_n \rightharpoonup f $
               weakly in  $W^{1,p'}(0,T;L^{p'}(\Omega ))$
                and $ f(0) = f(T)$, $f_n(0) = f_n (T)$ for any $n$.
\end{itemize}
Let $u_n$ be a solution of 
\begin{equation*}
\text{\rm (DNP)}_n 
   \begin{cases}
     ~\alpha_n( u ' _n(x,t)) - \Delta ^{a_n}_m u_n (x,t) = f_n (x,t)
         ~~& (x, t) \in \Omega \times (0,T), 
\\
     ~u  _n(x , t ) = 0 , 
         ~~& (x, t) \in \partial \Omega \times (0,T ), 
\\
     ~u_n(x, 0) = u_n(x,T), 
         ~~&  x \in  \Omega,
   \end{cases}
\end{equation*}
  whose existence is assured by Theorem \ref{Ex01}.
   Then there exist a subsequence $\{u _j \} _{j\in \N } := \{u _{n_j} \} _{j\in \N }$
and its limit $u$ such that
\begin{align*}
   u_j ~\to ~  u~~~&~~\text{ strongly in } C([0,T];L^{p}(\Omega ) ),
\\
              ~~~~~&~~\text{ weakly in } W^{1,p}(0,T;L^{p}(\Omega ) ),
\\
              ~~~~~&~~\ast\text{-weakly in } L^{\infty }(0,T;  W^{1,m}_0 (\Omega ) ),
\\
   \alpha (u'_j) ~\rightharpoonup  ~\alpha (u')
                ~~~&~~\text{ weakly in } L^{p'}(0,T; L^{p'}(\Omega) ),
\\
   \Delta^{a_j}_m u_j  ~\rightharpoonup  ~ \Delta^{a}_m u  
               ~~~&~~\text{ weakly in } L^{p'}(0,T; L^{p'}(\Omega)), 
\end{align*}
and $u $ is a solution to  {\rm (DNP)}. 
\end{theorem}


\appendix
\section{Appendix} 

\begin{lemma}
Let $E$ be a real reflexive Banach space and 
$\varphi :E \to [0, \infty ]$ be a proper lower semi-continuous convex functional.
Define 
\begin{equation*}
\Phi  (u) :=  \varphi (u) +\frac{\mu }{1 + \alpha } \varphi ^{1+\alpha } (u)
~~~~~u \in E  
\end{equation*}
for some  $\alpha , \mu > 0 $. Then the followings hold:
\begin{equation*}
 D(\Phi )  = D(\varphi ) ,
~~
 D(\partial_E \Phi ) =  D(\partial_E \varphi ) ,\\
~~
  \partial_E \Phi(u) = ( 1 + \mu \varphi(u)^{\alpha }) \partial_E \varphi(u) 
     \quad \forall u \in D(\partial_E \varphi) . 
\end{equation*}
\end{lemma}
\begin{proof}
First, 
$ D(\Phi ) = D(\varphi ) $  is trivial since $\varphi \geq 0 $.
Moreover, we can easily see that 
$ (1 + \mu \varphi ^{\alpha }) \partial_E \varphi \subset \partial_{E} \Phi $.
Indeed, for every $\LB u _1 , v _1 \RB  \in \partial_E \varphi $
and $u _2 \in D(\Phi ) = D(\varphi )$, we have
\begin{align*}
     &\LA ( 1 + \mu \varphi ^{\alpha } (u_1 )) v _1 , u _1 - u _  2 \RA_E 
\\
      & \geq ( 1 + \mu \varphi ^{\alpha } (u_1 ))   ( \varphi (u _1 ) - \varphi (u_2 )) 
\\
       & \geq  \varphi (u _1 ) - \varphi (u_2 )  
   	      + \mu \varphi ^{\alpha +1 } (u_1 )  
   	        -  \mu \LC \frac{\alpha }{\alpha +1 } 
   	              \varphi ^{\alpha \cdot \frac{\alpha +1}{\alpha } }(u_1 ) 
				     + \frac{1 }{\alpha +1 } \varphi ^{\alpha +1} (u_2 )  \RC 
\\
         & \geq \Phi (u _1 ) - \Phi (u_2 ).
\end{align*}
Hence we only have to check the maximality of 
  $ A := ( 1 + \mu \varphi ^{\alpha }) \partial_E \varphi $.

To this end, we rely on the following criterion 
(see Theorem 10.6 of \cite{Simons}).
\begin{lemma}\label{lemma:maximal} 
	  Let $E$ be a reflexive Banach space and 
	  $A :E \to 2 ^{E^{\ast}} $ be a monotone operator. 
Then $A$ is maximal monotone if and only if 
\begin{equation}\label{criterion:maximal}
   \begin{split}
 & G(A) + G(- F_E) = E \times E^\ast,
 \\ 
  &\text{i.e.,} \ 
    \forall (w,w^\ast) \in E \times E^\ast, \ \exists (u ,u^\ast) \in G(A) 
     \ \text{ such that} \quad 
         w^\ast \in u^\ast + F_E(u-w).
    \end{split}  
\end{equation}  	
\end{lemma} 
\noindent According to this lemma, 
  it suffices to show that for any $(w,w^\ast) \in E \times E^\ast$ there exists $u \in D(A)$ such that 
\begin{equation}\label{eq:A:maximal}
     F_E(u - w)  + ( 1 + \mu \varphi ^{\alpha }(u)) \partial_E \varphi (u) \ni w^\ast. 
\end{equation}
 To see this, for any $(w,w^\ast) \in E \times E^\ast$ and $ \lambda \geq 0$, 
  we consider the following auxiliary equation:
\begin{equation}
F_{E} (u_{\lambda } - w) + ( 1+ \lambda ) \partial_{E} \varphi (u_{\lambda } ) \ni w^\ast. 
\label{A2-01} 
\end{equation}
  Since $ (1 + \lambda) \partial_E \varphi$ is maximal monotone in $E \times E^\ast$, 
    Lemma \ref{lemma:maximal} assures that \eqref{A2-01} admits at least one solution. 
     In our setting, however, the uniqueness of solution is not ensured. 
Define  the set of all solutions of \eqref{A2-01} by $\mathcal{X}_\lambda \subset E$
       and  a set-valued mapping $\gamma : [0,\infty) \to 2^{[0,\infty)}$ by
\begin{equation*}
 \gamma(\lambda) :=  \{ ~\! \mu ~\! \varphi^\alpha ( u _\lambda) ~\! ; ~\! 
                         u_\lambda \in \mathcal{X}_\lambda ~\! \}.
\end{equation*}
   We are going to show that $\gamma$ has a fixed point $\lambda_0$. 
    It is easy to see that $u=u_{\lambda_0}$ gives a solution of 
      \eqref{eq:A:maximal}, which completes the proof.

   In order to show the existence of the fixed point of $\gamma$, we rely on Kakutani's fixed-point theorem.
    To do this, it suffices to verify the following facts:
\begin{itemize}
\item[{\#}1] There exists $L > 0 $ such that $\gamma ([0,L]) \subset [0,L ]$.

\item[{\#}2] The graph of $\gamma$ is closed, i.e., if $\lambda_n \to \lambda$ 
and $ y_n \to y$ with $y_n \in \gamma(\lambda_n)$, then $y \in \gamma(\lambda)$ holds true. 

\item [{\#}3] $\gamma(\lambda) $ is convex and non-empty for each $\lambda \in [0,\infty)$.

%
%
\end{itemize}

We first establish 
 a priori estimates for 
     $u_{\lambda }$ and $\varphi (u_{\lambda })$.
      Let $g_{\lambda }$ and $h_\lambda$ be the sections of $ \partial_E \varphi(u_\lambda)$ 
        and $F_E(u_\lambda - w)$ satisfying \eqref{A2-01}, namely,
\begin{equation}\label{def:g:lambda}
    g_{\lambda } := \frac{w^\ast - h_\lambda}{1 + \lambda } 
                     \in \partial_{E} \varphi (u_{\lambda }). 
\end{equation}
 By the definition of subdifferential,  
 \begin{equation}
  \varphi (v) -  \varphi (u_{\lambda })
  \geq \LA g_{\lambda } , v - u_{\lambda } \RA_{E} 
 = \LA \frac{w^\ast - h_\lambda}{1 + \lambda }  , v- u_{\lambda } \RA_{E} 
 \label{est:varphiu:lambda}
\end{equation}
holds for any fixed $v \in D(\varphi )$.
Hence
 $u_{\lambda } $ satisfies 
\begin{equation}
 2(1 + \lambda  ) \varphi (v) + |w-v|_E^2 +  2 ~\! |w^\ast|_{E^\ast} ( |v|_E + |u_\lambda|_E ) 
      \geq  2(1 + \lambda  ) \varphi (u_{\lambda })
              + |u_{\lambda } - w|^2_{E}, 
 \label{A2-02} 
\end{equation}
which implies the uniform boundedness of $|u_{\lambda }|_E$ and $\varphi(u_{\lambda })$
with respect to $\lambda $.

Let
  $u_{\nu } \in \mathcal{X} _{\nu}$, i.e., 
  $u_{\nu } \in D(\partial_{E} \varphi)$ 
be  a solution of \eqref{A2-01} 
	with $\lambda$ replaced by $\nu$.
  Substituting $u_{\nu} $ with $v $ in \eqref{est:varphiu:lambda}, we get 
\begin{equation}
  (1+ \lambda ) \varphi (u_{\nu }) - (1+ \lambda ) \varphi (u_{\lambda })
      \geq \LA w^\ast - h_\lambda , u_{\nu } - u_{\lambda } \RA_{E} .
  \label{A2-W01} 
\end{equation}
  Reversing the roles of $u_\lambda$ and $u_\nu$, we have 
\begin{equation}
  (1+ \nu ) \varphi (u_{\lambda  }) - (1+ \nu ) \varphi (u_{\nu })
      \geq \LA w^\ast - h_\nu , u_{\lambda  } - u_{\nu } \RA_{E} .
  \label{A2-W02} 
\end{equation}
By adding \eqref{A2-W01} and \eqref{A2-W02}, we obtain 
\begin{equation}
  ( \nu - \lambda ) \varphi (u_{\lambda  }) - ( \nu  -\lambda ) \varphi (u_{\nu })
     \geq \LA h_\lambda - h_\nu , u_{\lambda  } - u_{\nu } \RA_{E} .
  \label{A2-03} 
\end{equation}
Monotonicity of $F_{E}$ and \eqref{A2-03} yield
\begin{equation}\label{bound:above:phi}
   \varphi (u_0) \geq \varphi (u_{\lambda}) \geq \varphi (u_{\nu}) \quad 
\quad \quad 
                \text{ if } \ \nu \geq \lambda \geq 0, 
\end{equation}
where $u_0$, $u_\lambda$, and $u_\nu$ are arbitrary elements of $\mathcal{X}_0$,
 $\mathcal{X}_\lambda$,   and $\mathcal{X}_\nu$, respectively. 
Then \eqref{A2-03} also gives us 
\begin{equation}
  \LA h_\lambda - h_\nu , u_{\lambda  } - u_{\nu }  \RA_{E}
     \leq \varphi (u_0) |  \lambda -\nu |    
  \quad          \text{ for any } \lambda, \nu \geq 0.
\label{A2-04} 
\end{equation}
Multiplying the difference of two equations \eqref{A2-01} 
	with $\lambda=\lambda$ and $\lambda =\nu$ by $u_{\lambda } - u_{\nu }$ 
	    and using the monotonicity of $F_E$,
         we have 
\begin{equation*}
  \LA g_{\lambda } - g_{\nu } , u_{\lambda  } - u_{\nu }  \RA_{E}
   + \LA  \lambda g_{\lambda } - \nu g_{\nu } , u_{\lambda  } - u_{\nu }  \RA_{E}
     \leq 0
       ~~~\text{ for any } \lambda, \nu \geq 0.
\end{equation*}
 Here, if $\lambda \geq \nu $, 
\begin{align*}
 - \LA  \lambda g_{\lambda } - \nu g_{\nu }, u_{\lambda} - u_{\nu }  \RA_{E} 
    = &  - (\lambda - \nu ) 
           \LA   g_{\lambda }, u_{\lambda  } - u_{\nu }  \RA_{E}
            - \nu \LA   g_{\lambda} - g_{\nu}, u_{\lambda} - u_{\nu}  \RA_{E} 
  \\
      \leq &  (\lambda - \nu ) (\varphi (u_{\nu}  ) - \varphi (u_{\lambda})),
\end{align*}
and if $\nu \geq \lambda  $,
\begin{align*}
   - \LA  \lambda g_{\lambda } - \nu g_{\nu }, u_{\lambda} - u_{\nu}  \RA_{E} 
     = & - (\lambda - \nu ) 
           \LA   g_{\nu}, u_{\lambda} - u_{\nu}  \RA_{E}
              - \lambda  
                 \LA  g_{\lambda} -  g_{\nu}, u_{\lambda} - u_{\nu}  \RA_{E} 
  \\
           \leq &  (\nu - \lambda ) 
                     (\varphi (u_{\lambda}  ) - \varphi (u_{\nu})) .
\end{align*}
Hence 
\begin{equation}
   \LA g_{\lambda} - g_{\nu}, u_{\lambda} - u_{\nu}  \RA_{E}
       \leq  \varphi (u_0 ) |\lambda - \nu | 
          ~~~\text{ for any } \lambda, \nu \geq 0.
\label{A2-05} 
\end{equation}

Now from \eqref{A2-02}, we can assure {\#}1 with 
\begin{equation*}
  L := \mu \LC \varphi (v) + \frac{1}{2} |w - v|_{E}^2 
                     + | w^\ast|_{E^\ast} ( |v|_E + |w|_E + |w^\ast|_{E^\ast})  
                     \RC ^\alpha , 
\end{equation*}
where $v$ is an arbitrary element in $D(\varphi )$.

 In order to see {\#}2, let $\{ \lambda_n \}_{n \in \N}$ be a sequence of non-negative numbers
  which converges to $ \lambda \geq 0$ and let $ y_n \to y$ 
    with $y_n = \mu \varphi(u_{\lambda_n})^\alpha \in \gamma(\lambda_n)$.
     By  \eqref{def:g:lambda} and \eqref{A2-02}, we find that $|u_{\lambda_n}|_E, 
      |h_{\lambda_n}|_{E^\ast}, |g_{\lambda_n}|_{E^\ast} $ are all uniformly bounded. 
  Hence there exists a subsequence of $ \{\lambda_n\}$, denoted again by the same symbol, 
    such that 
    \begin{center}
  \begin{tabular}{ll}
   $\DS u_{\lambda_n}  \rightharpoonup u_\lambda$ \hspace{5mm} & weakly in $E$, \\[2mm]
   $\DS h_{\lambda_n}  \rightharpoonup h_\lambda $ \hspace{5mm} & weakly in  $E^\ast$, \\
   $\DS g_{\lambda_n}    \rightharpoonup g_\lambda = \frac{w^\ast - h_\lambda}{1 + \lambda}$ 
                                          \hspace{5mm}  & weakly in  $E^\ast$. 
  \end{tabular}
	\end{center}
%
  %
Thanks to Proposition \ref{Pro2.2}, 
   \eqref{A2-04} and  \eqref{A2-05}  with $\lambda = \lambda_n$ and $\nu = \lambda_m$ 
      imply that
\begin{equation}\label{converge:h:g}
   \begin{split}
     & h_{\lambda} \in F_{E} (u_\lambda), \quad 
         g_\lambda = \frac{w^\ast - h_\lambda}{1 + \lambda } 
           \in \partial_E \varphi (u_\lambda)
\\
    &  \LA h_{\lambda_n}, u_{\lambda_n} \RA_E  \quad \to 
         \quad  \LA h_\lambda, u_\lambda \RA_E 
           \quad \text{as} \ n \to \infty.
   \end{split}
\end{equation}
  Hence $u_\lambda \in \mathcal{X}_\lambda$, i.e., $u_\lambda$ is a solution 
   of \eqref{A2-01}.  Furthermore from \eqref{A2-W01} and \eqref{A2-W02} 
     with $\lambda=\lambda, \ \nu = \lambda_n $ together with \eqref{converge:h:g}, 
      we can derive 
\begin{align*}
  |\varphi (u_{\lambda_n}) -  \varphi (u_{\lambda}) |
     & \leq  | \LA h_{\lambda} - w^\ast, u_{\lambda_n} - u_{\lambda}  \RA_{E} | 
           +  | \LA h_{\lambda_n} - w^\ast, u_{\lambda} - u_{\lambda_n}  \RA_{E} |    
\\[2mm]
        & \to  0 ~~\text{ as } \  n\to \infty,
\end{align*}
   which implies that $ y = \mu \varphi(u_\lambda)^\alpha \in \gamma(\lambda)$. 
     Thus {\#}2 is verified. 
     
 To show {\#}3, we prepare the following lemma.
\begin{lemma}\label{lemma3}
 For each $\lambda \in [0,\infty)$, $\mathcal{X}_\lambda$ forms a non-empty closed convex 
   subset of $E$. Moreover for any $u_\lambda,  \bar{u}_\lambda 
     \in \mathcal{X}_\lambda$, it holds that 
\begin{equation}\label{identity:varphi:convex} 
  \varphi (\tau ~\! u_\lambda + (1-\tau )~\! \bar{u}_\lambda) 
     = \tau ~\! \varphi (u_\lambda) + (1-\tau ) ~\! \varphi(\bar{u}_\lambda) 
        \quad \forall \tau  \in (0,1).	
\end{equation}	
\end{lemma}
\begin{proof}[\sc Proof of Lemma \ref{lemma3}]
We first note that $\mathcal{X}_\lambda$ is not 
empty, since \eqref{A2-01} admits at least one solution. 
 We put
\begin{equation*}
   \varphi _1 (u) := \frac{1}{2} |u - w|^2_E + ( 1 + \lambda ) \varphi (u) 
      \quad u \in D(\varphi _1) := D ( \varphi).
\end{equation*}
Obviously, $\varphi _1 $ is a lower semi-continuous convex function and 
\begin{equation*}
   \partial_E \varphi _1 (u) = F_E(u - w) + ( 1 + \lambda )  \partial_E \varphi (u) 
        \quad u \in D(\partial _E \varphi _1 ) =D(\partial _E \varphi ).
\end{equation*}
Then 
$ u_\lambda,  \bar{u}_\lambda \in \mathcal{X}_\lambda$
is equivalent to 
$\LB u_{\lambda } , w^\ast \RB , \LB  \bar{u}_\lambda, w^\ast \RB \in  \partial_E \varphi _1 $, that is,
\begin{equation*}
 \varphi _1 (v) - \varphi _1 (u_\lambda) \geq \LA w^\ast, v - u_\lambda \RA_E ,~~~
 \varphi _1 (v) - \varphi _1 (\bar{u}_\lambda) \geq \LA w^\ast, v - \bar{u}_\lambda \RA_E 
\end{equation*}
for any $ v \in D(\varphi _1 )$.
These and  convexity of $\varphi _1$ yield
   \begin{align*}
   \varphi _1 (v) - \varphi _1 ( \tau ~\! u_\lambda + (1-\tau ) ~\! \bar{u}_\lambda) 
     & \geq \varphi _1 (v) - \tau \varphi _1  (u_\lambda) - (1-\tau ) \varphi _1 (\bar{u}_\lambda)  \\
      &   \geq \LA w^\ast, v - ( \tau ~\! u_\lambda + (1-\tau ) ~\! \bar{u}_\lambda) \RA_E ,
\end{align*}
   which means that $ \tau ~\! u_\lambda + (1-\tau ) ~\! \bar{u}_\lambda \in \mathcal{X}_\lambda$
for every $ u_\lambda,  \bar{u}_\lambda \in \mathcal{X}_\lambda$, 
    i.e., $\mathcal{X}_\lambda$ is convex. 
 
    Let $u^n_\lambda \in \mathcal{X}_\lambda$ converge to $u_\lambda$ strongly in $E$, 
      then letting $n \to \infty$ in 
\begin{equation*}
   \varphi _1 (v) - \varphi _1 (u^n_\lambda) \geq \LA w^\ast, v - u^n_\lambda \RA_E 
      \quad \forall v \in D(\varphi _1),
\end{equation*} 
  we have 
\begin{equation*}
  \varphi _1 (v) - \varphi _1 (u_\lambda) \geq \LA w^\ast, v - u_\lambda \RA_E 
  \quad \forall v \in D(\varphi _1 ),
\end{equation*}
  whence follows $u_\lambda \in \mathcal{X}_\lambda$, i.e., $\mathcal{X}_\lambda$ is 
   closed in $E$. 
   
   Let $ u_\lambda, \ \bar{u}_\lambda \in \mathcal{X}_\lambda$ and 
     let $g_\lambda, \ \bar{g}_\lambda$ be the sections of 
       $ \partial_E \varphi( u_\lambda), \ \partial_E \varphi(\bar{u}_\lambda)$ satisfying 
         \eqref{A2-01}.
   Then \eqref{A2-05} with $\nu = \lambda, \ u_\nu = \bar{u}_\lambda, \ 
      g_\nu = \bar{g}_\lambda$ implies 
\begin{equation*}
  \LA g_\lambda - \bar{g}_\lambda, u_\lambda - \bar{u}_\lambda \RA_E = 0, \ \text{i.e.,} \ 
     \LA \bar{g}_\lambda , u_\lambda - \bar{u}_\lambda \RA_E 
           = \LA g_\lambda , u_\lambda - \bar{u}_\lambda \RA_E.
\end{equation*}   
Therefore, 
\begin{align*}
  \tau ~\! (\varphi (u_\lambda) - \varphi (\bar{u}_\lambda)) 
   & = \tau \varphi (u_\lambda) + (1-\tau ) \varphi(\bar{u}_\lambda) - \varphi(\bar{u}_\lambda) 
\\
   & \geq \varphi( \tau u_\lambda + (1-\tau ) \bar{u}_\lambda) - \varphi (\bar{u}_\lambda)
\\
   &   \geq \LA \bar{g}_\lambda, \tau u_\lambda + (1-\tau ) \bar{u}_\lambda - \bar{u}_\lambda \RA_E 
\\
       &  = \tau \LA \bar{g}_\lambda, u_\lambda - \bar{u}_\lambda \RA_E
\\
  & 
	=  \tau \LA g_\lambda, u_\lambda - \bar{u}_\lambda \RA_E
  \\
   &
  \geq  \tau ~\! ( \varphi (u_\lambda) - \varphi(\bar{u}_\lambda)),
\end{align*}
 whence follows \eqref{identity:varphi:convex}.
\end{proof}

 
     We here claim that 
	$\Gamma_\lambda := \{  \varphi (u_\lambda)  ~; ~
                       u_\lambda \in \mathcal{X}_\lambda  \}$
                       is bounded, closed and convex in $\R$.
 The boundedness is obvious from \eqref{bound:above:phi} 
     and the closedness can be derived from the same arguments for the verification of 
       {\#2} above with $\lambda_n \equiv \lambda$.    
  For each $y_1, y_2 \in \Gamma_\lambda$, there exist $u_1, u_2 \in \mathcal{X}_\lambda$ 
    such that $ y_i = \varphi(u_i) \ (i=1,2)$. 
      Then from \eqref{identity:varphi:convex}, we have 
\begin{equation*}
 \tau y_1 + (1-\tau ) y_2 = \tau \varphi (u_1) + (1-\tau ) \varphi (u_2) 
                   = \varphi (\tau u_1 + (1-\tau ) u_2). 
\end{equation*} 
  Hence $\tau y_1 + (1-\tau ) y_2 \in \Gamma_\lambda$ follows
by Lemma \ref{lemma3}.
       Thus $\Gamma_\lambda$ is convex and 
there exist $ - \infty < a_\lambda \leq b_\lambda < + \infty$ 
       such that  
\begin{equation*}
   \Gamma_\lambda = [ a_\lambda, b_\lambda].
\end{equation*}
Immediately, we can conclude 
\begin{equation*}
  \gamma(\lambda) = [ ~\! \mu ~\!  a_\lambda^\alpha, ~\! \mu ~\! b_\lambda^\alpha ~\!], 
\end{equation*}     
  whence follows {\#}3.
\end{proof}


\address{
ABeam Consulting Ltd. \\
Marunouchi Eiraku Bldg.,\\
1-4-1 Marunouchi Chiyoda-ku, Tokyo,\\
Japan, 100-0005
}
{koike.4.19@gmail.com}
\address{
Department of Applied Physics, \\School of Science and Engineering, \\ 
Waseda University, \\
3-4-1, Okubo Shinjuku-ku, Tokyo,\\
Japan, 169-8555
}
{otani@waseda.jp}
\address{
Department of Integrated Science and Technology, \\
Faculty of Science and Technology, \\ 
Oita University,\\
700 Dannoharu, Oita City, Oita Pref., \\
Japan  870-1192
}
{shunuchida@oita-u.ac.jp}
\end{document}